\documentclass[11pt]{amsart}
\usepackage{amsmath,amssymb,amsthm,latexsym,cite,cancel}
\usepackage[small]{caption}
\usepackage{graphicx,wasysym,overpic,tikz,color}
\usepackage[colorlinks=true,urlcolor=blue,
citecolor=blue,linkcolor=blue,linktocpage,pdfpagelabels,
bookmarksnumbered,bookmarksopen]{hyperref}
\usepackage[english]{babel}
\usepackage{units}
\usepackage{enumitem}
\usepackage[left=2.1cm,right=2.1cm,top=2.71cm,bottom=2.71cm]{geometry}
\usepackage{float,bm}
\makeatletter
\def\@seccntformat#1{%
	\protect\textup{\protect\@secnumfont
		\csname the#1\endcsname\protect\enspace}}
\makeatother
\newtheorem{theorem}{Theorem}[section]
\newtheorem{definition}[theorem]{Definition}
\newtheorem{proposition}[theorem]{Proposition}
\newtheorem{lemma}[theorem]{Lemma}

\newtheorem{example}[theorem]{Example}

\newcommand{\abs}[1]{\lvert#1\rvert}
\newcommand{\norm}[1]{\lVert#1\rVert}

\tikzstyle{nodo}=[circle,draw,fill,inner sep=0pt,minimum size=%
1.5mm]

\numberwithin{equation}{section}

\title[Sharp Born--Infeld Sobolev inequality and extremal functions]
{Sharp Born--Infeld Sobolev inequality and extremal functions on the lattice graph $\mathbb Z^N$}

\author[C. Ji]{Chao Ji}

\address[C. Ji]{\newline\indent
	School of Mathematics
	\newline\indent
	East China University of Science and Technology
	\newline\indent
	Shanghai 200237, PR China }
\email{\href{mailto:jichao@ecust.edu.cn}{jichao@ecust.edu.cn}}

\author[K. Sheng]{Kai Sheng}

\address[K. Sheng]{\newline\indent
	School of Mathematics
	\newline\indent
	East China University of Science and Technology
	\newline\indent
	Shanghai 200237, PR China}

\email{\href{mailto:shengkai2001@outlook.com}{shengkai2001@outlook.com}}
\date{\today}

\begin{document}

\begin{abstract}

In this paper, we study the sharp Sobolev-type inequality associated with
the Born--Infeld energy on the lattice graph \(\mathbb Z^N\), \(N\geq3\):
\[
\frac12\sum_{x\in\mathbb Z^N}\sum_{y\sim x}
\left(1-\sqrt{1-|\nabla_{xy}u|^2}\right)
\geq
C_{N,\alpha}
\left(\sum_{x\in\mathbb Z^N}|u(x)|^\alpha\right)^{\frac{N}{N+\alpha}},
\]
for \(u\in
D^{1,2}(\mathbb Z^N)\cap\ell^\alpha(\mathbb Z^N)
\) satisfying
\(|\nabla_{xy}u|\leq1\) for every \(x\sim y\), where
\(C_{N,\alpha}\) denotes the optimal constant.
We determine the exact positivity threshold and prove that
\(C_{N,\alpha}>0\) if and only if
\(\alpha\geq2^*:=2N/(N-2)\).
At the Sobolev critical exponent \(\alpha=2^*\), we identify the optimal constant as
\[
C_{N,2^*}=\frac12\mathcal S_2,
\]
where \(\mathcal S_2\) is the optimal discrete Sobolev constant, and show
that it is not attained. In the supercritical regime, there exists
\(\varepsilon_0=\varepsilon_0(N)>0\) such that \(C_{N,\alpha}\) is attained
for every \(2^*<\alpha<2^*+\varepsilon_0\), with a nonnegative Schwarz symmetric extremal function. The main compactness difficulties stem from the lack of a suitable scaling on \(\mathbb Z^N\) and from the nonhomogeneity of the Born--Infeld energy. We overcome them by combining
discrete Schwarz rearrangement with a \(Q_1\) discrete-to-continuum
comparison and by establishing a strict separation as the \(\ell^\alpha\)-norm tends to infinity.

\end{abstract}

\maketitle
\section{Introduction}

The Lorentz--Minkowski energy, also known as the electrostatic
Born--Infeld energy, gives rise to the following Sobolev-type
inequality, first studied by Bonheure, De Coster, and Derlet
\cite{BonheureDeCosterDerlet2012}:
\begin{equation}\label{eq:continuous-lorentz-minkowski-inequality}
\int_{\mathbb R^N}
\left(
1-\sqrt{1-|\nabla u|^2}
\right)\,dx
\geq
\widetilde C_{N,\alpha}
\left(
\int_{\mathbb R^N}|u|^\alpha\,dx
\right)^{\frac{N}{N+\alpha}},
\qquad u\in\mathcal X,
\end{equation}
where \(N\geq3\), \(\alpha>2\), and
\[\mathcal X
=
\left\{
u\in D^{1,2}(\mathbb R^N):
\nabla u\in L^\infty(\mathbb R^N,\mathbb R^N),\
\|\nabla u\|_{L^\infty(\mathbb R^N)}\leq1
\right\}.\]
The authors proved that
\eqref{eq:continuous-lorentz-minkowski-inequality} holds with a positive
constant if and only if \(\alpha\geq2^*\). Moreover, the optimal constant
is not attained at the Sobolev critical exponent \(\alpha=2^*\), whereas
it is attained for every \(\alpha>2^*\).
More recently, Bieganowski, Ikoma, and Mederski
\cite{BieganowskiIkomaMederski2025} obtained a new variational proof of inequality~\eqref{eq:continuous-lorentz-minkowski-inequality} in the supercritical regime
and characterized its optimal constant in terms of the ground-state energy level of the corresponding
Born--Infeld equation.

The significance of this extremal problem goes beyond the attainment of the optimal constant. Indeed, in \cite{BonheureDeCosterDerlet2012}, for \(N\geq3\) and \(\alpha>2^*\), the study of extremal functions for
\eqref{eq:continuous-lorentz-minkowski-inequality} plays a key role in the construction of infinitely many radial solutions to the quasilinear equation
\[
-\operatorname{div}\left(
\frac{\nabla u}{\sqrt{1-|\nabla u|^2}}
\right)
=
|u|^{\alpha-2}u
\qquad \text{in } \mathbb R^N.
\]
From a physical viewpoint, Born--Infeld theory was introduced as a
nonlinear modification of Maxwell electrodynamics in order, in
particular, to remove the infinite self-energy of a point charge:
whereas the Coulomb electric field becomes singular at the charge, the
Born--Infeld electrostatic field of a point charge remains bounded and
has finite self-energy
\cite{BornInfeld1934,Kiessling2004,Kiessling2012}. Geometrically, the associated mean-curvature operator is the Lorentzian
counterpart of the minimal-surface operator. In Euclidean space, the
equation for a minimal graph is
\[
\operatorname{div}\left(
\frac{\nabla u}{\sqrt{1+|\nabla u|^2}}
\right)=0.
\]
For comparison, the Lorentz--Minkowski metric is
\[
ds_L^2=dx_1^2+\cdots+dx_N^2-dt^2.
\]
The sign change in the ambient metric is reflected in the area
functional for graphs: its integrand is
\(\sqrt{1+|\nabla u|^2}\) in Euclidean space, whereas for a spacelike
graph in Lorentz--Minkowski space it is
\(\sqrt{1-|\nabla u|^2}\). Consequently, the zero mean-curvature
equation is the maximal hypersurface equation. In this setting, the classical
Bernstein-type theory for entire maximal spacelike hypersurfaces goes
back to Calabi and Cheng--Yau
\cite{Calabi1970,ChengYau1976}. More generally, the same operator with
a nonzero right-hand side gives prescribed mean-curvature equations
for spacelike graphs; see, for instance, Bartnik and Simon
\cite{BartnikSimon1982}. For further developments on existence, multiplicity, ground states,
and regularity for Born--Infeld and Lorentz--Minkowski mean-curvature
equations, we refer to \cite{Azzollini2014,Azzollini2016,BonheureColasuonnoFoldes2019,
BonheureIacopetti2019,BonheureIacopetti2023,MederskiPomponio2023,
BaldelliMederskiPomponio2025,MederskiZeng2026,
ByeonIkomaMalchiodiMari2026,BieganowskiIkomaMederski2025}
and the references therein.

In recent years, functional inequalities on discrete graphs have
received considerable attention and substantial progress has been made
in developing discrete counterparts of classical inequalities. Such
inequalities provide important analytic tools for the study of
extremal and variational problems, as well as discrete nonlinear equations. In particular, Hua and Li \cite{HuaLi2021} studied extremal
functions for discrete Sobolev and Hardy--Littlewood--Sobolev
inequalities on \(\mathbb Z^N\). For \(1\leq p<N\) and
\(q\geq p^*:=Np/(N-p)\), the discrete Sobolev inequality takes the form
\begin{equation}\label{eq:intro-discrete-sobolev-inequality}
\lVert u\rVert_{\ell^q(\mathbb Z^N)}
\leq
\overline C_{N,p,q}
\lVert u\rVert_{D^{1,p}(\mathbb Z^N)},
\qquad
u\in D^{1,p}(\mathbb Z^N).
\end{equation}
They introduced a discrete concentration--compactness principle and then combined it with the homogeneity of \eqref{eq:intro-discrete-sobolev-inequality} to prove the existence of extremal functions for
\eqref{eq:intro-discrete-sobolev-inequality} in the supercritical range
\(q>p^*\). Very recently, He and Ji \cite{HeJi2026} resolved the Sobolev
critical case for \(p=2\): for \(N\geq3\), they proved that the optimal
constant in the critical discrete Sobolev inequality with \(q=2^*\) is
attained on \(\mathbb Z^N\). In a related direction, Hajaiej, Han and Hua
\cite{Hajaiej2026} developed a higher-dimensional discrete Schwarz
rearrangement on \(\mathbb Z^N\), together with the corresponding P\'olya--Szeg\H{o}-type inequalities, providing a useful compactness tool for extremal problems on lattice graphs. For further results on functional
inequalities and related extremal problems in discrete settings, we
refer to
\cite{HuangLiYin2015,Gupta2023,GuptaSteinerberger2024,
HuaLiMunch2025,HanLi2026}
and the references therein. 

Motivated by the works mentioned above, we study the discrete counterpart of the
Born--Infeld Sobolev inequality and its associated extremal problem on
\(\mathbb Z^N\). First, we
determine the exact range of exponents for which the optimal constant is
positive. Second, at the Sobolev critical exponent \(\alpha=2^*\), we
identify the optimal constant in terms of the sharp discrete Sobolev
constant and prove that it is not attained. Third, in the supercritical
regime, we establish attainment for all \(\alpha>2^*\) sufficiently close
to \(2^*\), and show that an extremal function can be chosen nonnegative
and Schwarz symmetric. 

More precisely, in this paper we consider the following inequality:
\begin{equation}\label{eq:sobolev-type-inequality}
\frac12\sum_{x\in\mathbb Z^N}\sum_{y\sim x}
\left(
1-\sqrt{1-|\nabla_{xy}u|^2}
\right)
\geq
C_{N,\alpha}
\left(
\sum_{x\in\mathbb Z^N}|u(x)|^\alpha
\right)^{\frac{N}{N+\alpha}},
\qquad u\in X\cap\ell^\alpha(\mathbb Z^N),
\end{equation}
where \(N\geq3\), \(\alpha>2\), and
\[
X
:=
\left\{
u\in D^{1,2}(\mathbb Z^N):
|\nabla_{xy}u|\leq1
\ \text{for every }x\sim y
\right\}.
\]
The restriction \(u\in X\cap\ell^\alpha(\mathbb Z^N)\) guarantees that the
right-hand side of \eqref{eq:sobolev-type-inequality} is finite for every
\(\alpha>2\). The issue arises in the subcritical range
\(2<\alpha<2^*\), where the embedding
\(D^{1,2}(\mathbb Z^N)\hookrightarrow\ell^\alpha(\mathbb Z^N)\)
fails. For \(\alpha\geq2^*\), the discrete Sobolev inequality gives
\(X\subset\ell^\alpha(\mathbb Z^N)\), so the restriction is redundant in this case. In the sequel, \(x\sim y\) means that \(x\) and \(y\)
are adjacent in \(\mathbb Z^N\). For an adjacent ordered pair \((x,y)\),
we write \(\nabla_{xy}u:=u(y)-u(x)\) for the discrete edge difference.
We denote by \(\ell^2(\mathbb Z^N)\) the space of
\(\ell^2\)-summable functions on
\(\mathbb Z^N\) and by \(D^{1,2}(\mathbb Z^N)\) the completion of
finitely supported functions in the \(D^{1,2}\)-norm; see
Section~\ref{2} for details.

The optimal constant in \eqref{eq:sobolev-type-inequality} is given by
\[C_{N,\alpha}
:=
\inf_{u\in (X\cap\ell^\alpha(\mathbb Z^N))\setminus\{0\}}
\frac{
	\displaystyle
	\frac12\sum_{x\in\mathbb Z^N}\sum_{y\sim x}
	\left(
	1-\sqrt{1-|\nabla_{xy}u|^2}
	\right)
}{
	\displaystyle
	\left(
	\sum_{x\in\mathbb Z^N}|u(x)|^\alpha
	\right)^{\frac{N}{N+\alpha}}
}.\]

For \(2\leq p<N\), we recall that the optimal constant in the discrete
Sobolev inequality is given by
\[
\mathcal S_p
:=
\inf_{u\in D^{1,p}(\mathbb Z^N)\setminus\{0\}}
\frac{
\displaystyle
\frac12\sum_{x\in\mathbb Z^N}\sum_{y\sim x}
|\nabla_{xy}u|^p
}{
\displaystyle
\|u\|_{\ell^{p^*}(\mathbb Z^N)}^p
}
>0,
\qquad
p^*:=\frac{Np}{N-p}.
\]
For every \(p\geq2\), we also define
\[
\kappa_p
:=
\inf_{0<s\leq1}
\frac{1-\sqrt{1-s}}{s^{p/2}}
>0.
\]
By the definition of \(\kappa_p\), we have
\begin{equation}\label{eq:kappa-lower-bound}
1-\sqrt{1-s}
\geq
\kappa_p s^{p/2},
\qquad
s\in[0,1].
\end{equation}

For later use, we write
\[
 p_\alpha:=\frac{N\alpha}{N+\alpha}.
\]
Notice that \(\alpha\geq2^*\) is equivalent to \(p_\alpha\geq2\), and
\(p_\alpha^*=\alpha\). Moreover, \eqref{eq:kappa-lower-bound} and
\(s^{p/2}\leq s\) on \([0,1]\) show that
\[
 \kappa_p\geq\frac12,\qquad p\geq2.
\]

Now we state our main results. To determine when inequality~\eqref{eq:sobolev-type-inequality} holds for some positive constant, we first establish the following estimate.

\begin{theorem}\label{thm:subcritical}
Let \(N\geq3\) and \(\alpha>2\). Then
\(C_{N,\alpha}>0\) if and only if
\(\alpha\geq2^*:=\frac{2N}{N-2}\).
More precisely, if \(\alpha\geq2^*\), then
\[
0<
\kappa_{p_\alpha}\mathcal S_{p_\alpha}
\leq
C_{N,\alpha}
<
+\infty.
\]
\end{theorem}

After identifying the range \(\alpha\geq2^*\) in which
\(C_{N,\alpha}\) is positive, we turn to the corresponding extremal
problem. We first consider the Sobolev critical case \(\alpha=2^*\), where
we characterize the optimal constant \(C_{N,2^*}\) in terms of the discrete
Sobolev constant \(\mathcal S_2\) and prove that it is not attained.

\begin{theorem}\label{thm:main}
Let \(N\geq3\). Then
\[
C_{N,2^*}
=
\frac12\mathcal S_2,
\]
and \(C_{N,2^*}\) is not attained.
\end{theorem}

It is worth pointing out that this nonattainment is specific to the
Born--Infeld functional. Indeed, the optimal constant \(\mathcal S_2\)
in the underlying critical discrete Sobolev inequality is attained
\cite{HeJi2026}, whereas Theorem~\ref{thm:main} shows that
\(C_{N,2^*}=\frac12\mathcal S_2\) is not attained. Thus the failure of
attainment does not come from the critical discrete Sobolev problem
itself, but from the strictly nonquadratic structure of the Born--Infeld
energy. A key step in the critical nonattainment argument of Bonheure,
De Coster, and Derlet \cite{BonheureDeCosterDerlet2012} is the scaling technique
\[
u(x)\longmapsto u_t(x):=t^{N/\alpha}u(tx),
\]
which preserves the \(L^\alpha\)-norm. Under the assumption that the optimal constant is attained, they exploit minimality along this scaling family to derive a contradiction. However, this scaling technique is not available in the discrete setting. To overcome this
difficulty, we exploit the following critical feature of our problem. At the threshold \(\alpha=2^*\), we have
\(p_\alpha=2\). This identity shows that the Sobolev
critical exponent is exactly the point at which the power in the discrete
Sobolev estimate matches the quadratic small-gradient behavior of
\(1-\sqrt{1-|\nabla u|^2}\).

At the Sobolev critical exponent, we use
\(1-\sqrt{1-r^2}=\frac12r^2+O(r^4)\) as \(r\to0\) to obtain the upper bound $
C_{N,2^*}\leq \frac12\mathcal S_2$. Indeed, for a fixed
\(u\in D^{1,2}(\mathbb Z^N)\), the scaled functions \(\varepsilon u\) lie in
\(X\) for sufficiently small \(\varepsilon>0\) (see Section~\ref{sec:critical-subcritical}), and
\[
\frac{1}{2\varepsilon^2}
 \sum_{x\in\mathbb Z^N}\sum_{y\sim x}
 \left(1-\sqrt{1-\varepsilon^2|\nabla_{xy}u|^2}\right)
 \longrightarrow \frac12\|u\|_{D^{1,2}(\mathbb Z^N)}^2.
\]
The reverse bound follows from the elementary inequality
\(1-\sqrt{1-s}\geq s/2\). Since this inequality is strict for \(s>0\), the sharp critical constant cannot be attained by a nonzero function. 

In contrast to the Sobolev critical case, in the Sobolev supercritical case we
obtain the existence of extremal functions. Specifically, we prove
that \(C_{N,\alpha}\) is attained when \(\alpha>2^*\) is sufficiently
close to \(2^*\).

\begin{theorem}\label{achieved}
Let \(N\geq3\). There exists \(\varepsilon_0=\varepsilon_0(N)>0\) such
that, for every \(2^*<\alpha<2^*+\varepsilon_0\), the optimal constant
\(C_{N,\alpha}\) is attained by a nonnegative Schwarz symmetric function.
\end{theorem}

Theorem~\ref{achieved} is the main compactness result of the paper. For its proof, the scaling technique commonly used in the analysis on \(\mathbb R^N\) is no longer applicable on \(\mathbb Z^N\). Two distinct
compactness issues arise. The first is the usual loss of compactness under
lattice translations. At each fixed \(\ell^\alpha\)-constraint level, this difficulty
can be handled by discrete Schwarz rearrangement. The second is specific to
the nonhomogeneity of the Born--Infeld energy
\[
J(u):=\frac12\sum_{x\in\mathbb Z^N}\sum_{y\sim x}
\left(1-\sqrt{1-|\nabla_{xy}u|^2}\right).
\]
Moreover, in contrast with the homogeneous discrete
Sobolev problem, the \(\ell^\alpha\)-constraint cannot be normalized by
amplitude scaling. Consequently, along a minimizing sequence for
\(C_{N,\alpha}\), the \(\ell^\alpha\)-norm may tend either to zero or
to infinity. We therefore separate the proof into compactness at a fixed
constraint level and the genuinely nonhomogeneous analysis of the regimes
where the constraint level tends to \(0\) or \(+\infty\).

Discrete Schwarz rearrangement itself is an established
tool and has already been used in related lattice extremal and constrained
problems; see \cite{HanLi2026,HeJiTao2025}. Here it yields the Born--Infeld
P\'olya--Szeg\H{o}-type inequality
\[
 J(u^*)\leq J(u),
\]
which permits us to work with Schwarz symmetric functions and recover
compactness for each fixed constraint. The new global step begins by
introducing, for \(t>0\),
\[
m(t):=\inf\left\{J(u):u\in X,\
\|u\|_{\ell^\alpha(\mathbb Z^N)}=t\right\},
\]
and rewriting the optimal constant as
\[
C_{N,\alpha}=\inf_{t>0}\frac{m(t)}{t^{p_\alpha}}.
\]
The remaining issue is to prevent a minimizing sequence \(\{t_n\}\) for \(m(t)/t^{p_\alpha}\) from
escaping to \(0\) or \(+\infty\). The behavior as \(t\to0^+\) follows
from the discrete Sobolev inequality.
To control the large-\(t\) regime, we use the \(Q_1\)-interpolation
considered in \cite{OrtnerShapeev2012} to compare lattice functions
with the continuous Sobolev inequality. Related discrete-to-continuum
comparisons have recently been developed by He and Ji
\cite{HeJi2026} in their study of singular limits on lattice graphs.
There, the edge length tends to zero, whereas in the present problem
the edge length is fixed and the continuum comparison is used to
analyze the large-\(t\) regime \(t\to\infty\).

Through this discrete-to-continuum comparison, the argument exploits a strict separation between
the discrete critical Sobolev level governing the variational problem and
the continuous Sobolev level arising in the limit \(t\to\infty\). In
particular, He and Ji \cite{HeJi2026} proved the strict comparison
\(\mathcal S_2<S_2^{\mathbb R^N}\).
Our proof of
Theorem~\ref{achieved} is independent of this work and gives a
self-contained quantitative separation. The restriction
\(2^*<\alpha<2^*+\varepsilon_0\) enters precisely at this comparison step:
the discrete test-function upper bound is compared with the large-\(t\)
lower bound, and the strict separation is obtained by using
\(p_\alpha\to2\) as \(\alpha\downarrow2^*\). Thus compactness at each fixed
constraint level holds for every \(\alpha>2^*\), whereas the present
argument rules out the possibility \(t\to\infty\) only when \(\alpha\) is
sufficiently close to \(2^*\). A suitable discrete test function places \(C_{N,\alpha}\)
strictly below an intermediate level, while the \(Q_1\)
discrete-to-continuum comparison places the asymptotic level as \(t\to\infty\)
strictly above the same level. Consequently,
\[
C_{N,\alpha}
<
\liminf_{t\to\infty}\frac{m(t)}{t^{p_\alpha}},
\]
which is the strict variational gap that rules out the possibility that the constraint parameter tends to \(+\infty\). Together with the estimate as \(t\to0^+\), this yields
the attainment asserted in Theorem~\ref{achieved}. See
Section~\ref{sec:supercritical} for details.

\noindent\textbf{Open problem.}
In the continuous problem \eqref{eq:continuous-lorentz-minkowski-inequality},
the optimal constant is attained for every \(\alpha>2^*\). In the discrete
problem, the fixed-constraint compactness argument applies for every
\(\alpha>2^*\), but the strict variational gap needed to exclude
\(t\to\infty\) is established here only in the near-critical range.
It therefore remains open whether \(C_{N,\alpha}\) is attained on
\(\mathbb Z^N\) for every \(\alpha>2^*\).

The paper is organized as follows. Section~\ref{2} collects the discrete
Sobolev, Born--Infeld, and rearrangement preliminaries. Section~\ref{sec:critical-subcritical}
proves the positivity threshold and the critical sharp constant. In
Section~\ref{sec:supercritical} we study the constrained levels \(m(t)\),
establish compactness at a fixed constraint level, develop the \(Q_1\) discrete-to-continuum
estimates, and prove the strict variational gap leading to Theorem~\ref{achieved}.

\noindent\textbf{Notation:}
Unless otherwise stated, we use the following notation:
\begin{itemize}
    \item
   \(C,C_1,C_2,\ldots\) denote positive constants whose values are not
    relevant.

    \item
    \(2^*:=2N/(N-2)\) if \(N\geq3\).

    \item
    For a set \(A\), \(\mathbf 1_A\) denotes the characteristic function
    of \(A\).
\end{itemize}

\section{Preliminaries}\label{2}

In this section, we introduce the notation and collect the preliminary
results used in the proofs of our main theorems. We first recall the lattice graph, some related function spaces and the discrete Sobolev inequality,
then study the admissible set \(X\) and the energy functional \(J\), and
finally recall the discrete Schwarz rearrangement and the compactness
property needed in the Sobolev supercritical case.

\subsection{The lattice graph and basic notation}

We regard the \(N\)-dimensional integer lattice graph
\(\mathbb Z^N\) as the graph \(G=(V,E)\), where
\[
	V:=\mathbb Z^N
\]
is the vertex set and
\[
	E
	:=
	\left\{
		(x,y)\in\mathbb Z^N\times\mathbb Z^N:
		\sum_{i=1}^N\abs{x^{(i)}-y^{(i)}}=1
	\right\}
\]
is the edge set. Throughout this paper, we regard \(E\) as a subset of
\(V\times V\). Thus, if \((x,y)\in E\), then \((y,x)\in E\) as well,
and these two ordered pairs represent the two orientations of the same
undirected edge. Two vertices \(x,y\in\mathbb Z^N\) are called
neighbors, denoted by \(x\sim y\), if \((x,y)\in E\). The
combinatorial distance \(d\) is defined by
\[
	d(x,y)
	:=
	\inf\left\{
		k:x=x_0\sim x_1\sim\cdots\sim x_k=y
	\right\}.
\]
We denote the space of functions on \(\mathbb Z^N\) by
\(C(\mathbb Z^N)\). For \(u\in C(\mathbb Z^N)\), its support set is
defined as
\(\operatorname{supp}(u)
:=
\{x\in\mathbb Z^N:u(x)\neq0\}\).
Let \(C_c(\mathbb Z^N)\) be the set of all functions with finite support.
For \(u\in C(\mathbb Z^N)\), the \(\ell^p\)-norm of \(u\) is defined by
\[
	\norm{u}_{\ell^p(\mathbb Z^N)}
	:=
	\begin{cases}
		\displaystyle
		\left(
		\sum_{x\in\mathbb Z^N}\abs{u(x)}^p
		\right)^{1/p},
		& 1\leq p<\infty,\\[3mm]
		\displaystyle
		\sup_{x\in\mathbb Z^N}\abs{u(x)},
		& p=\infty.
	\end{cases}
\]
The \(\ell^p(\mathbb Z^N)\) space is defined as
\[
	\ell^p(\mathbb Z^N)
	:=
	\left\{
	u\in C(\mathbb Z^N):
	\norm{u}_{\ell^p(\mathbb Z^N)}<\infty
	\right\}.
\]
When there is no confusion, we write
\(\norm{u}_p:=\norm{u}_{\ell^p(\mathbb Z^N)}\).
For \(u,v\in C(\mathbb Z^N)\), define the gradient form by
\[
	\Gamma(u,v)(x)
	:=
	\frac12\sum_{y\sim x}
	\bigl(u(y)-u(x)\bigr)\bigl(v(y)-v(x)\bigr).
\]
Let \(\Gamma(u):=\Gamma(u,u)\) and
\[
	|\nabla u|(x)
	:=
	\sqrt{\Gamma(u)(x)}
	=
	\left(
	\frac12\sum_{y\sim x}|u(y)-u(x)|^2
	\right)^{1/2}.
\]
For an adjacent ordered pair \((x,y)\), we recall the difference
notation
\[
	\nabla_{xy}u:=u(y)-u(x).
\]
For \(u\in C(\mathbb Z^N)\), define the edge gradient
\(\nabla_Eu:E\to\mathbb R\) by
\[
	(\nabla_Eu)(x,y)
	:=
	\nabla_{xy}u,
	\qquad
	(x,y)\in E.
\]
For \(1\leq p<\infty\), the \(D^{1,p}\)-norm of \(u\) is given by
\[
	\norm{u}_{D^{1,p}(\mathbb Z^N)}
	:=
	\left(
	\frac12\sum_{x\in\mathbb Z^N}
	\sum_{y\sim x}
	\abs{\nabla_{xy}u}^p
	\right)^{1/p}.
\]
In particular,
\[
	\norm{u}_{D^{1,2}(\mathbb Z^N)}
	=
	\norm{|\nabla u|}_{\ell^2(\mathbb Z^N)}.
\]
The edgewise Lipschitz seminorm is given by
\[
	\norm{\nabla_Eu}_{\ell^\infty(E)}
	:=
	\sup_{(x,y)\in E}
	\abs{\nabla_{xy}u}.
\]
\(D^{1,p}(\mathbb Z^N)\) is the completion of
\(C_c(\mathbb Z^N)\) in the \(D^{1,p}\)-norm.

For \(1\leq p<N\), set
\[
	p^*:=\frac{Np}{N-p}.
\]
By the discrete Sobolev inequality on \(\mathbb Z^N\)
(see, for instance, \cite{HuaLi2021}), for every \(q\geq p^*\),
there exists a constant \(\overline C_{N,p,q}>0\) such that
\begin{equation}\label{eq:discrete-sobolev-inequality}
	\lVert u\rVert_{\ell^q(\mathbb Z^N)}
	\leq
	\overline C_{N,p,q}
	\lVert u\rVert_{D^{1,p}(\mathbb Z^N)}
	\qquad
	\text{for all }u\in D^{1,p}(\mathbb Z^N).
\end{equation}

We also record the elementary formula for the constant \(\kappa_p\)
defined in the Introduction:
\[
\kappa_p
=
\begin{cases}
\displaystyle \frac12,
& p=2,\\[2mm]
\displaystyle
\frac{p-1}{p}
\left(
\frac{(p-1)^2}{p(p-2)}
\right)^{\frac{p-2}{2}},
& p>2.
\end{cases}
\]
In particular, \(\kappa_p\to\frac12\) as \(p\to2^+\).

We shall repeatedly use the following standard cutoff characterization of
\(D^{1,p}(\mathbb Z^N)\). It is included to make explicit a point needed in
the rearrangement and interpolation arguments below.

\begin{lemma}\label{lem:finite-energy-characterization}
Let \(1\leq p<N\) and \(p^*=Np/(N-p)\). Suppose that
\(u:\mathbb Z^N\to\mathbb R\) satisfies
\[
 u\in\ell^{p^*}(\mathbb Z^N)
 \qquad\text{and}\qquad
 \frac12\sum_{x\in\mathbb Z^N}\sum_{y\sim x}
 |\nabla_{xy}u|^p<\infty.
\]
Then \(u\in D^{1,p}(\mathbb Z^N)\).
\end{lemma}

\begin{proof}
Choose cutoff functions \(\eta_R:\mathbb Z^N\to[0,1]\) such that
\(\eta_R=1\) on \([-R,R]^N\cap\mathbb Z^N\), \(\eta_R=0\) outside
\([-2R,2R]^N\cap\mathbb Z^N\), and
\[
 |\eta_R(x)-\eta_R(y)|\leq \frac{C}{R}
 \qquad\text{whenever }x\sim y.
\]
Set \(u_R:=\eta_Ru\in C_c(\mathbb Z^N)\). For an edge \(x\sim y\),
\[
 \nabla_{xy}(u-u_R)
 =(1-\eta_R(y))\nabla_{xy}u
 +u(x)(\eta_R(x)-\eta_R(y)).
\]
Hence, with \(A_R:=([-2R-1,2R+1]^N\setminus[-R+1,R-1]^N)
\cap\mathbb Z^N\),
\[
 \|u-u_R\|_{D^{1,p}}^p
 \leq C\!\sum_{\substack{x\sim y\\ \max\{|x|_\infty,|y|_\infty\}\geq R-1}}
 |\nabla_{xy}u|^p
 +\frac{C}{R^p}\sum_{x\in A_R}|u(x)|^p.
\]
The first term tends to zero by summability of the edge \(p\)-energy. For
the second term, H\"older's inequality and \(|A_R|\leq CR^N\) yield
\[
 \frac1{R^p}\sum_{x\in A_R}|u(x)|^p
 \leq C
 \left(\sum_{x\in A_R}|u(x)|^{p^*}\right)^{p/p^*}
 \longrightarrow0,
\]
because \(1-p/p^*=p/N\) and \(u\in\ell^{p^*}(\mathbb Z^N)\).
Therefore \(u_R\to u\) in the \(D^{1,p}\)-norm, proving the claim.
\end{proof}

\subsection{The admissible set and the energy functional}

We define
\[X
:=
\left\{
u\in D^{1,2}(\mathbb Z^N):
\abs{\nabla_{xy}u}\leq1
\ \text{for every }x\sim y
\right\}.\]

We first record two elementary properties of the admissible set \(X\).

\begin{lemma}
	Fix \(\alpha>2\). The set \(X\) has the following properties:
	\begin{enumerate}
		\item[\textup{(i)}]
		For every \(t>0\), the set
		\(\{u\in X:\norm{u}_{\alpha}=t\}\) is not empty.

		\item[\textup{(ii)}]
		The set \(X\) is weakly closed in
		\(D^{1,2}(\mathbb Z^N)\).
	\end{enumerate}
\end{lemma}

\begin{proof}
	First, we show that
	\(\{u\in X:\norm{u}_{\alpha}=t\}\) is not empty for every \(t>0\).
	For \(n\in\mathbb N\), let
	\[
		Q_n
		:=
		\left[-\frac n2,\frac n2\right)^N
		\cap\mathbb Z^N
	\]
	and define
	\[
		u_n
		:=
		t n^{-\frac{N}{\alpha}}
		\mathbf 1_{Q_n}.
	\]
	Since \(\abs{Q_n}=n^N\), we have
	\[
		\norm{u_n}_{\alpha}^{\alpha}
		=
		\sum_{x\in\mathbb Z^N}\abs{u_n(x)}^\alpha
		=
		n^N t^\alpha n^{-N}
		=
		t^\alpha,
	\]
	and hence \(\norm{u_n}_{\alpha}=t\). Moreover,
	\(u_n\in C_c(\mathbb Z^N)\) and
	\(\abs{\nabla_{xy}u_n}\leq tn^{-N/\alpha}\) for every \(x\sim y\).
	For \(n\) sufficiently large, \(tn^{-N/\alpha}\leq1\), so that
	\(u_n\in X\). Hence
	\(\{u\in X:\norm{u}_{\alpha}=t\}\) is not empty.

	Next, we show that \(X\) is weakly closed in
	\(D^{1,2}(\mathbb Z^N)\). Let \(\{u_n\}\subset X\) and suppose that
	\[
		u_n\rightharpoonup u
		\qquad
		\text{in }D^{1,2}(\mathbb Z^N).
	\]
	Since the embedding
	\[
		D^{1,2}(\mathbb Z^N)
		\hookrightarrow
		\ell^{2^*}(\mathbb Z^N)
	\]
	is continuous, we have
	\[
		u_n\rightharpoonup u
		\qquad
		\text{in }\ell^{2^*}(\mathbb Z^N).
	\]
	For each fixed \(x\in\mathbb Z^N\), the coordinate map
	\(v\mapsto v(x)\) is a continuous linear functional on
	\(\ell^{2^*}(\mathbb Z^N)\). Hence
	\[
		u_n(x)\longrightarrow u(x)
		\qquad\text{for every }x\in\mathbb Z^N.
	\]
	It follows that, for every \(x\sim y\),
	\[
		\nabla_{xy}u_n
		\longrightarrow
		\nabla_{xy}u.
	\]
	Since \(u_n\in X\), we have
	\[
		\abs{\nabla_{xy}u_n}\leq1
		\qquad
		\text{for every }x\sim y.
	\]
	Passing to the limit, we obtain
	\[
		\abs{\nabla_{xy}u}\leq1
		\qquad
		\text{for every }x\sim y.
	\]
	Hence \(u\in X\), and therefore \(X\) is weakly closed in
	\(D^{1,2}(\mathbb Z^N)\).
\end{proof}

We define
\[	J(u)
:=
\frac12\sum_{x\in\mathbb Z^N}\sum_{y\sim x}
\left(
1-\sqrt{1-\abs{\nabla_{xy}u}^2}
\right),
\qquad
u\in X.\]

We next record the coercivity and weak lower semicontinuity of \(J\).

\begin{lemma}
	The functional \(J\) has the following properties:
	\begin{enumerate}
		\item[\textup{(i)}]
		\(J\) is coercive on \(X\).

		\item[\textup{(ii)}]
		\(J\) is weakly lower semicontinuous on \(X\).
	\end{enumerate}
\end{lemma}

\begin{proof}
A direct calculation gives
	\[
		\frac12s^2
		\leq
		1-\sqrt{1-s^2}
		\leq
		s^2
		\qquad
		\text{for every }s\in[0,1].
	\]
	Hence
	\[
		\frac12
		\norm{u}_{D^{1,2}(\mathbb Z^N)}^2
		\leq
		J(u)
		\leq
		\norm{u}_{D^{1,2}(\mathbb Z^N)}^2
		\qquad
		\text{for every }u\in X,
	\]
	showing \textup{(i)}.

To prove \textup{(ii)}, let \(\{u_n\}\subset X\) and suppose
	that
	\[
		u_n\rightharpoonup u
		\qquad
		\text{in }D^{1,2}(\mathbb Z^N).
	\]
	Since \(X\) is weakly closed in \(D^{1,2}(\mathbb Z^N)\), we have
	\(u\in X\). By the continuity of the embedding
	\[
		D^{1,2}(\mathbb Z^N)
		\hookrightarrow
		\ell^{2^*}(\mathbb Z^N),
	\]
	we have
	\[
		u_n\rightharpoonup u
		\qquad
		\text{in }\ell^{2^*}(\mathbb Z^N).
	\]
	As above, the coordinate maps are continuous on
	\(\ell^{2^*}(\mathbb Z^N)\), and hence
	\[
		u_n(x)\longrightarrow u(x)
		\qquad\text{for every }x\in\mathbb Z^N.
	\]
	It follows that
	\[
		\nabla_{xy}u_n
		\longrightarrow
		\nabla_{xy}u
		\qquad
		\text{for every }x\sim y.
	\]
	By Fatou's lemma, we obtain
\[
		\begin{aligned}
			J(u)
			&=
			\frac12\sum_{x\in\mathbb Z^N}\sum_{y\sim x}
			\left(
			1-\sqrt{1-\abs{\nabla_{xy}u}^2}
			\right)
			\\
			&\leq
			\liminf_{n\to\infty}
			\frac12\sum_{x\in\mathbb Z^N}\sum_{y\sim x}
			\left(
			1-\sqrt{1-\abs{\nabla_{xy}u_n}^2}
			\right)
			\\
			&=
			\liminf_{n\to\infty}J(u_n).
		\end{aligned}
	\]
\end{proof}

\subsection{Discrete Schwarz rearrangement}

The discrete Schwarz rearrangement developed in
\cite{Hajaiej2026} is crucial for the proof of our main results.
We therefore briefly recall its definition and some related properties.

Let
\[
	C_0^+(\mathbb Z^N)
	:=
	\left\{
	u\in C(\mathbb Z^N):
	u\geq0,\
	\abs{
	\left\{
	x\in\mathbb Z^N:u(x)>t
	\right\}
	}
	<\infty
	\ \text{for every }t>0
	\right\}
\]
be the set of functions that are admissible for discrete Schwarz
rearrangement. For \(u\in C_0^+(\mathbb Z^N)\), its discrete Schwarz
rearrangement is denoted by \(u^*:=R_{\mathbb Z^N}u\). Since
\(D^{1,2}(\mathbb Z^N)\hookrightarrow\ell^{2^*}(\mathbb Z^N)\), every
\(u\in D^{1,2}(\mathbb Z^N)\) satisfies
\(\abs{\{x:\abs{u(x)}>t\}}<\infty\) for every \(t>0\). Thus, for a
real-valued \(u\in D^{1,2}(\mathbb Z^N)\), we may write
\(u^*:=R_{\mathbb Z^N}\abs{u}\).
We define
\[
	\Sigma
	:=
	\left\{
	u\in D^{1,2}(\mathbb Z^N):u\geq0,\ u=u^*
	\right\}.
\]
For more details on the definition and properties of the discrete Schwarz
rearrangement, we refer to \cite{Hajaiej2026}.
We recall that
\begin{equation}\label{eq:rearrangement-norm-preserving}
	\norm{u^*}_q
	=
	\norm{u}_q,
	\qquad
	1\leq q\leq\infty.
\end{equation}

\begin{lemma}\label{lem:compact-symmetric-embedding}
	For every \(\alpha>2^*\), the embedding
	\(\Sigma\hookrightarrow\ell^\alpha(\mathbb Z^N)\) is compact.
\end{lemma}

\begin{proof}
	Let \(\{u_n\}\subset\Sigma\) be bounded in
	\(D^{1,2}(\mathbb Z^N)\). Since the embedding
	\[
		D^{1,2}(\mathbb Z^N)
		\hookrightarrow
		\ell^{2^*}(\mathbb Z^N)
	\]
	is continuous, \(\{u_n\}\) is bounded in
	\(\ell^{2^*}(\mathbb Z^N)\). Hence
	\cite[Theorem~4.16]{Hajaiej2026} yields, up to a subsequence,
	\[
		u_n\longrightarrow u
		\qquad
		\text{in }\ell^\alpha(\mathbb Z^N)
	\]
	for every \(\alpha>2^*\).
\end{proof}

\section{Proofs of Theorems~\ref{thm:subcritical} and
\ref{thm:main}}\label{sec:critical-subcritical}

In this section, we prove Theorems~\ref{thm:subcritical}
and~\ref{thm:main}. We first establish a lower bound for
\(C_{N,\alpha}\), which yields positivity in the Sobolev critical and
supercritical ranges. We then treat the Sobolev critical case, identify
\(C_{N,2^*}\), and prove its nonattainment.

\begin{lemma}\label{lem:estimate}
Let \(N\geq3\) and \(\alpha\geq2^*\). Then
\[
C_{N,\alpha}
\geq \kappa_{p_\alpha}
	\mathcal S_{p_\alpha}
>0.
\]
\end{lemma}

\begin{proof}
For every \(u\in X\setminus\{0\}\), the discrete Sobolev inequality gives
\(u\in\ell^\alpha(\mathbb Z^N)\). Moreover, since
\(p_\alpha\geq2\) and \(|\nabla_{xy}u|\leq1\), we obtain
\[
\sum_{x\in\mathbb Z^N}\sum_{y\sim x}
|\nabla_{xy}u|^{p_\alpha}
\leq
\sum_{x\in\mathbb Z^N}\sum_{y\sim x}
|\nabla_{xy}u|^2
<\infty.
\]
As \(p_\alpha^*=\alpha\),
Lemma~\ref{lem:finite-energy-characterization} gives
\(u\in D^{1,p_\alpha}(\mathbb Z^N)\).

Then, by \eqref{eq:kappa-lower-bound} and the definition of
\(\mathcal S_{p_\alpha}\),
\[
\begin{aligned}
	J(u)
	&\geq
	\kappa_{p_\alpha}
	\frac12\sum_{x\in\mathbb Z^N}\sum_{y\sim x}
	\abs{\nabla_{xy}u}^{p_\alpha}
	\\
	&\geq
	\kappa_{p_\alpha}
	\mathcal S_{p_\alpha}
	\norm{u}_{\ell^\alpha(\mathbb Z^N)}
	^{p_\alpha}.
\end{aligned}
\]
	
Therefore,
\[
C_{N,\alpha} 
\geq 
\kappa_{p_\alpha}
\mathcal S_{p_\alpha}
>0. 
\]
\end{proof}

We now prove Theorem~\ref{thm:subcritical}.

\begin{proof}
By Lemma~\ref{lem:estimate}, if
\(\alpha\geq2^*\), then
\[
C_{N,\alpha} 
\geq 
\kappa_{p_\alpha}
	\mathcal S_{p_\alpha}
>0. 
\]
On the other hand, taking \(u=\mathbf 1_{\{0\}}\), we have
\(u\in X\) and \(\|u\|_{\ell^\alpha(\mathbb Z^N)}=1\). Moreover,
with our edge-counting convention, \(J(u)=2N\). Hence
\(C_{N,\alpha}\leq2N<+\infty\).

It remains to consider the case \(2<\alpha<2^*\). To prove that
\(C_{N,\alpha}=0\), fix any
\(u\in X\cap C_c(\mathbb Z^N)\setminus\{0\}\) and, for
\(k\in\mathbb N\), define \(u_k:=u/k\). Since \(u\in X\), we
clearly have \(u_k\in X\). Moreover,
\(\norm{u_k}_{\ell^\alpha(\mathbb Z^N)}
=\frac1k\norm{u}_{\ell^\alpha(\mathbb Z^N)}\).
By \(1-\sqrt{1-s}\leq s\) for \(s\in[0,1]\), we obtain
\[
\begin{aligned} 
	\frac{J(u_k)} 
	{\norm{u_k}_{\ell^\alpha(\mathbb Z^N)} 
		^{p_\alpha}}
	&\leq 
	\frac{ 
		k^{-2} 
		\displaystyle 
		\frac12\sum_{x\in\mathbb Z^N}\sum_{y\sim x} 
		\abs{\nabla_{xy}u}^{2} 
	}{ 
		k^{-p_\alpha}
		\norm{u}_{\ell^\alpha(\mathbb Z^N)} 
		^{p_\alpha}
	} 
	\\ 
	&= 
	k^{p_\alpha-2}
	\frac{ 
		\norm{u}_{D^{1,2}(\mathbb Z^N)}^2 
	}{ 
		\norm{u}_{\ell^\alpha(\mathbb Z^N)} 
		^{p_\alpha}
	}. 
\end{aligned} 
\]
Since \(\alpha<2^*\), we have \(p_\alpha<2\).
Therefore,
\(k^{p_\alpha-2}\longrightarrow0\) as \(k\to\infty\).
Therefore,
\[
	C_{N,\alpha}=0.
\]
\end{proof}

We now prove Theorem~\ref{thm:main}.

\begin{proof}
By Lemma~\ref{lem:estimate}, for \(\alpha=2^*\), we have
\(\frac{N2^*}{N+2^*}=p_{2^*}=2\) and \(\kappa_2=1/2\), and hence
\[
	C_{N,2^*}\geq\frac12\mathcal S_2.
\]

To prove that
\[
C_{N,2^*}\leq\frac12\mathcal S_2,
\]
let \(u\in D^{1,2}(\mathbb Z^N)\setminus\{0\}\) and set
\(u_\varepsilon:=\varepsilon u\).
Since
\[
	\norm{\nabla_Eu}_{\ell^\infty(E)}
	\leq
	\norm{u}_{D^{1,2}(\mathbb Z^N)},
\]
we have \(u_\varepsilon\in X\) for all sufficiently small
\(\varepsilon>0\). Moreover,
\[
	\varepsilon^{-2}
	\left(
	1-\sqrt{1-\abs{\nabla_{xy}u_\varepsilon}^{2}}
	\right)
	=
	\frac{\abs{\nabla_{xy}u}^{2}}
	{1+\sqrt{1-\varepsilon^2\abs{\nabla_{xy}u}^{2}}}
	\leq
	\abs{\nabla_{xy}u}^{2}.
\]
By the Lebesgue dominated convergence theorem,
\[
	\lim_{\varepsilon\to0}
	\varepsilon^{-2}J(u_\varepsilon)
	=
	\frac14
	\sum_{x\in\mathbb Z^N}\sum_{y\sim x}
	\abs{\nabla_{xy}u}^{2}
	=
	\frac12\norm{u}_{D^{1,2}(\mathbb Z^N)}^2.
\]
Therefore,
\[
	C_{N,2^*}
	\leq
	\frac{J(u_\varepsilon)}
	{\norm{u_\varepsilon}_{\ell^{2^*}(\mathbb Z^N)}^2}
	=
	\frac{\varepsilon^{-2}J(u_\varepsilon)}
	{\norm{u}_{\ell^{2^*}(\mathbb Z^N)}^2}.
\]
Letting \(\varepsilon\to0\), we have
\[
	C_{N,2^*}
	\leq
	\frac12
	\frac{
		\displaystyle
		\norm{u}_{D^{1,2}(\mathbb Z^N)}^2
	}{
		\displaystyle
		\norm{u}_{\ell^{2^*}(\mathbb Z^N)}^2
	}.
\]
Since \(u\in D^{1,2}(\mathbb Z^N)\setminus\{0\}\) is arbitrary,
\[
	C_{N,2^*}\leq\frac12\mathcal S_2.
\]
Thus
\[
	C_{N,2^*}=\frac12\mathcal S_2.
\]

It remains to prove that \(C_{N,2^*}\) is not attained. Arguing
indirectly, suppose that \(C_{N,2^*}\) is attained by a function
\(u\in X\setminus\{0\}\). Since \(\mathbb Z^N\) is connected, if
\(\nabla_{xy}u=0\) for every \(x\sim y\), then \(u\) is constant; as
\(u\in\ell^{2^*}(\mathbb Z^N)\), this would force \(u=0\). Hence there
exist \(x_0\sim y_0\) such that \(\nabla_{x_0y_0}u\neq0\). Since
\[
	1-\sqrt{1-s}>\frac12s,
	\qquad s\in(0,1],
\]
we obtain
\[
	J(u)
	>
	\frac14
	\sum_{x\in\mathbb Z^N}\sum_{y\sim x}
	\abs{\nabla_{xy}u}^{2}
	=
	\frac12\norm{u}_{D^{1,2}(\mathbb Z^N)}^2.
\]
Consequently,
\[
	C_{N,2^*}
	=
	\frac{J(u)}
	{\norm{u}_{\ell^{2^*}(\mathbb Z^N)}^2}
	>
	\frac12
	\frac{
		\displaystyle
		\norm{u}_{D^{1,2}(\mathbb Z^N)}^2
	}{
		\displaystyle
		\norm{u}_{\ell^{2^*}(\mathbb Z^N)}^2
	}
	\geq
	\frac12\mathcal S_2,
\]
which is the desired contradiction. Therefore, \(C_{N,2^*}\) is not attained.
\end{proof}

\section{The Sobolev supercritical case}\label{sec:supercritical}

In this section, we prove Theorem~\ref{achieved}. As explained in the
Introduction, to obtain the boundedness of minimizing
sequences, we consider the following constrained minimization problem.
For \(\alpha>2^*\) and \(t>0\), define
\[
	m(t)
	:=
	\inf\left\{
	J(u):
	u\in X,\
	\norm{u}_{\ell^\alpha(\mathbb Z^N)}=t
	\right\}.
\]
We first study the constrained problem and characterize
\(C_{N,\alpha}\) in terms of the quotient
\(m(t)/t^{p_\alpha}\). To show that the
infimum of this quotient is attained, we prove its lower
semicontinuity and study its behavior as \(t\to0^+\) and
\(t\to\infty\). The estimate as \(t\to\infty\) is obtained by using the
\(Q_1\)-interpolation.

\subsection{The constrained problem}

To obtain the compactness needed for the constrained problem,
we first establish a P\'olya--Szeg\H{o}-type inequality for the
functional \(J\).

\begin{lemma}\label{lem:energy-rearrangement}
For every \(u\in X\),
\[
	u^*\in X
	\qquad\text{and}\qquad
	J(u^*)\leq J(u).
\]
\end{lemma}

\begin{proof}
For every \(x\sim y\),
\[
	\abs{\nabla_{xy}\abs{u}}
	=
	\bigl|\abs{u(y)}-\abs{u(x)}\bigr|
	\leq
	\abs{\nabla_{xy}u}.
\]
Since the map \(s\mapsto1-\sqrt{1-s^2}\) is increasing on \([0,1]\),
we have \(\abs{u}\in X\) and \(J(\abs{u})\leq J(u)\). Moreover, by
definition, \(u^*=(\abs{u})^*\). It is therefore enough to consider
\(u\geq0\).

Following the regularization argument used in the proof of
\cite[Lemma~4.3]{BonheureDeCosterDerlet2012}, we handle the square-root
singularity at the endpoint \(1\) by approximating the Born--Infeld energy
density with nonnegative, increasing, and convex functions. The subsequent
rearrangement argument is carried out in the discrete setting using the
discrete Riesz inequality of \cite{Hajaiej2026}.
For \(n\geq2\), define \(H_n:[0,\infty)\to[0,\infty)\) by
\[
	H_n(b)
	:=
	\begin{cases}
		\displaystyle
		1-\sqrt{1-b},
		& 0\leq b<1-\dfrac1{n^2},
		\\[2mm]
		\displaystyle
		1-\dfrac1n
		+\dfrac n2
		\left(
			b-1+\dfrac1{n^2}
		\right),
		& b\geq1-\dfrac1{n^2}.
	\end{cases}
\]
Set \(F_n(s):=H_n(s^2)\) for \(s\geq0\).
The function \(H_n\) is nonnegative and increasing. It is convex
because the derivative of \(1-\sqrt{1-b}\) increases up to \(n/2\) at
\(b=1-1/n^2\), which equals the slope of the affine continuation. Since
\(s\mapsto s^2\) is convex and \(H_n\) is convex and nondecreasing,
\(F_n=H_n\circ(s\mapsto s^2)\) is also convex and nondecreasing.
Hence, by
\cite[Proposition~5.4(ii)]{Hajaiej2026}, the function
\[
	\mathcal G_n(s,t)
	:=
	-F_n\bigl(\abs{s-t}\bigr)
	=
	-H_n\bigl(\abs{s-t}^2\bigr)
\]
is supermodular.

Let \(K(r):=\mathbf 1_{[0,1]}(r)\). To apply
\cite[Theorem~\(5.7'\)]{Hajaiej2026}, for \(k\in\mathbb N\), set
\[
	u_k:=\left(u-\frac1k\right)_+.
\]
We first show that
\begin{equation}\label{eq:Gn-integrability}
\left|
	\sum_{x,y\in\mathbb Z^N}
	\mathcal G_n(u_k(x),0)K(d(x,y))
\right|
+
\left|
	\sum_{x,y\in\mathbb Z^N}
	\mathcal G_n(0,u_k(y))K(d(x,y))
\right|
<\infty.
\end{equation}
Since \(u\in\ell^{2^*}(\mathbb Z^N)\), the function \(u_k\) has
finite support. Moreover, \(\mathcal G_n(0,0)=0\) and
\(K(d(x,y))=0\) whenever \(d(x,y)>1\). Hence,
\[
\left|
	\sum_{x,y\in\mathbb Z^N}
	\mathcal G_n(u_k(x),0)K(d(x,y))
\right|
<\infty
\qquad\text{and}\qquad
\left|
	\sum_{x,y\in\mathbb Z^N}
	\mathcal G_n(0,u_k(y))K(d(x,y))
\right|
<\infty.
\]
We conclude \eqref{eq:Gn-integrability}.

The truncation commutes with Schwarz rearrangement. Indeed, for every
\(s>0\),
\[
 \{(u-1/k)_+>s\}=\{u>s+1/k\},
\]
so the level sets of \((u-1/k)_+^*\) and \((u^*-1/k)_+\) have the same
cardinality and the same rearranged shape. Hence
\[
	u_k^*=\left(u^*-\frac1k\right)_+.
\]
Applying \cite[Theorem~\(5.7'\)]{Hajaiej2026} with
\(G=\mathcal G_n\), \(H=K\), and \(u=v=u_k\), and using
\(\mathcal G_n(s,s)=0\), we obtain
\[
	\sum_{x\in\mathbb Z^N}\sum_{y\sim x}
	H_n\bigl(\abs{\nabla_{xy}u_k}^2\bigr)
	\geq
	\sum_{x\in\mathbb Z^N}\sum_{y\sim x}
	H_n\bigl(\abs{\nabla_{xy}u_k^*}^2\bigr).
\]
For every \(x\sim y\),
\[
	H_n\bigl(\abs{\nabla_{xy}u_k}^2\bigr)
	\uparrow
	H_n\bigl(\abs{\nabla_{xy}u}^2\bigr),
	\qquad
	H_n\bigl(\abs{\nabla_{xy}u_k^*}^2\bigr)
	\uparrow
	H_n\bigl(\abs{\nabla_{xy}u^*}^2\bigr)
\]
as \(k\to\infty\). To see the monotonicity explicitly, after
ordering two nonnegative endpoint values as \(a\geq b\), the quantity
\((a-c)_+-(b-c)_+\) increases to \(a-b\) as \(c\downarrow0\).
Since \(H_n\) is increasing and all terms are nonnegative, letting
\(k\to\infty\), we obtain
\begin{equation}\label{eq:edge-approximating-energy}
	\sum_{x\in\mathbb Z^N}\sum_{y\sim x}
	H_n\bigl(\abs{\nabla_{xy}u}^2\bigr)
	\geq
	\sum_{x\in\mathbb Z^N}\sum_{y\sim x}
	H_n\bigl(\abs{\nabla_{xy}u^*}^2\bigr).
\end{equation}
Let
\[
	\mathsf E
	:=
	\left\{
		(x,y)\in\mathbb Z^N\times\mathbb Z^N:
		x\sim y
	\right\}
\]
and define
\[
	A
	:=
	\left\{
		(x,y)\in\mathsf E:
		\abs{\nabla_{xy}u}\geq\frac12
	\right\}.
\]
Since \(u\in D^{1,2}(\mathbb Z^N)\),
\[
	\frac14\abs{A}
	\leq
	\sum_{(x,y)\in A}
	\abs{\nabla_{xy}u}^2
	<\infty.
\]
Hence \(A\) is finite.

Define \(h:\mathsf E\to[0,\infty)\) by
\[
	h(x,y)
	:=
	\begin{cases}
		1,
		& (x,y)\in A,
		\\
		\abs{\nabla_{xy}u}^2,
		& (x,y)\notin A.
	\end{cases}
\]
Then \(\sum_{(x,y)\in\mathsf E}h(x,y)<\infty\).
Moreover, for every \(n\geq2\) and \((x,y)\in\mathsf E\),
\[
	0\leq H_n\bigl(\abs{\nabla_{xy}u}^2\bigr)\leq h(x,y).
\]
Indeed,
\[
	H_n\bigl(\abs{\nabla_{xy}u}^2\bigr)
	\leq
	\begin{cases}
		1=h(x,y),
		& (x,y)\in A,
		\\[1mm]
		1-\sqrt{1-\abs{\nabla_{xy}u}^2}
		\leq
		\abs{\nabla_{xy}u}^2=h(x,y),
		& (x,y)\in\mathsf E\setminus A.
	\end{cases}
\]
Consequently,
\begin{equation}\label{eq:approximating-energy-uniform-bound}
	\sup_{n\geq2}
	\sum_{x\in\mathbb Z^N}\sum_{y\sim x}
	H_n\bigl(\abs{\nabla_{xy}u}^2\bigr)
	<\infty.
\end{equation}

We next show that \(u^*\in X\). Suppose that
\(\abs{\nabla_{x_0y_0}u^*}^2>1\) for some \(x_0\sim y_0\). Then
\[
	H_n\bigl(\abs{\nabla_{x_0y_0}u^*}^2\bigr)
	\longrightarrow+\infty,
\]
which contradicts \eqref{eq:edge-approximating-energy} and
\eqref{eq:approximating-energy-uniform-bound}. Therefore,
\[
	\abs{\nabla_{xy}u^*}\leq1,
	\qquad x\sim y.
\]
Moreover, \(H_2(b)\geq b/2\) for every \(b\geq0\).
Taking \(n=2\) in \eqref{eq:edge-approximating-energy}, we obtain
\[
\begin{aligned}
	\frac12
	\sum_{x\in\mathbb Z^N}\sum_{y\sim x}
	\abs{\nabla_{xy}u^*}^2
	&\leq
	\sum_{x\in\mathbb Z^N}\sum_{y\sim x}
	H_2\bigl(\abs{\nabla_{xy}u^*}^2\bigr)
	\\
	&\leq
	\sum_{x\in\mathbb Z^N}\sum_{y\sim x}
	H_2\bigl(\abs{\nabla_{xy}u}^2\bigr)
	<\infty.
\end{aligned}
\]
Moreover, by
\eqref{eq:rearrangement-norm-preserving}, we have
\(u^*\in\ell^{2^*}(\mathbb Z^N)\). 
The preceding estimate gives finite edge \(2\)-energy. Applying
Lemma~\ref{lem:finite-energy-characterization} with \(p=2\), we conclude
that \(u^*\in D^{1,2}(\mathbb Z^N)\); together with
\(|\nabla_{xy}u^*|\leq1\), this yields \(u^*\in X\).

For every \(x\sim y\),
\[
H_n\bigl(\abs{\nabla_{xy}u}^2\bigr)
\longrightarrow
1-\sqrt{1-\abs{\nabla_{xy}u}^2}
\quad\text{as }n\to\infty.
\]
By the Lebesgue dominated convergence theorem, we obtain
\[
\begin{aligned}
	&\lim_{n\to\infty}
	\sum_{x\in\mathbb Z^N}\sum_{y\sim x}
	H_n\bigl(\abs{\nabla_{xy}u}^2\bigr)
	\\
	&\qquad=
	\sum_{x\in\mathbb Z^N}\sum_{y\sim x}
	\left(
		1-\sqrt{1-\abs{\nabla_{xy}u}^2}
	\right).
\end{aligned}
\]
Since \(u^*\in X\), repeating the same argument yields
\[
\begin{aligned}
	&\lim_{n\to\infty}
	\sum_{x\in\mathbb Z^N}\sum_{y\sim x}
	H_n\bigl(\abs{\nabla_{xy}u^*}^2\bigr)
	\\
	&\qquad=
	\sum_{x\in\mathbb Z^N}\sum_{y\sim x}
	\left(
		1-\sqrt{1-\abs{\nabla_{xy}u^*}^2}
	\right).
\end{aligned}
\]
We conclude that
\[
	\sum_{x\in\mathbb Z^N}\sum_{y\sim x}
	\left(
		1-\sqrt{1-\abs{\nabla_{xy}u}^2}
	\right)
	\geq
	\sum_{x\in\mathbb Z^N}\sum_{y\sim x}
	\left(
		1-\sqrt{1-\abs{\nabla_{xy}u^*}^2}
	\right).
\]
\end{proof}

Using Lemmas~\ref{lem:energy-rearrangement} and
\ref{lem:compact-symmetric-embedding}, we obtain a minimizer of
\(m(t)\).

\begin{proposition}\label{prop:constrained-minimizer}
Let \(\alpha>2^*\) and \(t>0\). Then \(m(t)\) is attained by a
nonnegative Schwarz symmetric function.
\end{proposition}

\begin{proof}
Let \(\{v_n\}\subset X\) be a minimizing sequence for \(m(t)\) such that
\[
	\norm{v_n}_{\ell^\alpha(\mathbb Z^N)}=t,
	\qquad
	J(v_n)\longrightarrow m(t).
\]
By Lemma~\ref{lem:energy-rearrangement},
\[
	v_n^*\in X,
	\qquad
	\norm{v_n^*}_{\ell^\alpha(\mathbb Z^N)}=t,
	\qquad
	J(v_n^*)\leq J(v_n).
\]
Since \(J\) is coercive, \(\{v_n^*\}\) is bounded in
\(D^{1,2}(\mathbb Z^N)\). Passing to a subsequence if necessary,
\[
	v_n^*\rightharpoonup u
	\qquad\text{in }D^{1,2}(\mathbb Z^N)
\]
for some \(u\in X\). Since \(\alpha>2^*\), the compactness of the
embedding \(\Sigma\hookrightarrow\ell^\alpha(\mathbb Z^N)\) implies that
\[
	v_n^*\longrightarrow u
	\qquad\text{in }\ell^\alpha(\mathbb Z^N).
\]
In particular, \(v_n^*(x)\to u(x)\) for every \(x\in\mathbb Z^N\).
Since each \(v_n^*\) is nonnegative and Schwarz symmetric,
\cite[Proposition~4.15]{Hajaiej2026} implies that \(u\) is also
nonnegative and Schwarz symmetric. Moreover,
\(\norm{u}_{\ell^\alpha(\mathbb Z^N)}=t\).
Since
\[
m(t)\leq J(v_n^*)\leq J(v_n)\longrightarrow m(t),
\]
we have \(J(v_n^*)\to m(t)\). By the weak lower semicontinuity of \(J\),
\[
m(t)
\leq
J(u)
\leq
\liminf_{n\to\infty}J(v_n^*)
=
m(t).
\]
Thus \(J(u)=m(t)\). Moreover, \(u\) is a nonnegative Schwarz symmetric
minimizer of \(m(t)\).
\end{proof}

Combining the definition of \(m(t)\) with that of \(C_{N,\alpha}\),
we obtain the following characterization of \(C_{N,\alpha}\).

\begin{proposition}\label{prop:quotient-constrained-relation}
Let \(\alpha>2^*\). Then
\[
C_{N,\alpha}
=
\inf_{t>0}
\frac{m(t)}
{t^{p_\alpha}}.
\]
\end{proposition}

\begin{proof}
For every \(t>0\), by the definition of \(C_{N,\alpha}\), we have
\[
C_{N,\alpha} 
\leq 
\frac{J(v)} 
{\norm{v}_{\ell^\alpha(\mathbb Z^N)}^{p_\alpha}}
\]
for every \(v\in X\setminus\{0\}\) satisfying
\(\norm{v}_{\ell^\alpha(\mathbb Z^N)}=t\).
For every \(t>0\),
\[
C_{N,\alpha} 
\leq 
\inf_{\substack{v\in X\\ 
		\norm{v}_{\ell^\alpha(\mathbb Z^N)}=t}} 
\frac{J(v)} 
{\norm{v}_{\ell^\alpha(\mathbb Z^N)} 
	^{p_\alpha}}
= 
\frac{m(t)} 
{t^{p_\alpha}}.
\]
Hence
\[
C_{N,\alpha}
\leq
\inf_{t>0}
\frac{m(t)}
{t^{p_\alpha}}.
\]
On the other hand, let \(u\in X\setminus\{0\}\) be arbitrary and set
\[
	t:=\norm{u}_{\ell^\alpha(\mathbb Z^N)}>0.
\]
By the definition of \(m(t)\),
\[
	J(u)\geq m(t).
\]
Therefore,
\[
\frac{J(u)} 
{\norm{u}_{\ell^\alpha(\mathbb Z^N)}^{p_\alpha}}
\geq 
\frac{m(t)}{t^{p_\alpha}}
\geq 
\inf_{s>0} 
\frac{m(s)}{s^{p_\alpha}}.
\]
Since \(u\in X\setminus\{0\}\) is arbitrary, we have
\[
C_{N,\alpha}
\geq
\inf_{s>0}
\frac{m(s)}
{s^{p_\alpha}}.
\]
Combining the two inequalities, we conclude that
\[
C_{N,\alpha}
=
\inf_{t>0}
\frac{m(t)}
{t^{p_\alpha}}.
\]
\end{proof}

Finally, we prove the lower semicontinuity of the above quotient.

\begin{lemma}\label{lem:mt-lower-semicontinuity}
Let \(\alpha>2^*\). Then
\[
t\longmapsto
\frac{m(t)}
{t^{p_\alpha}}
\]
is lower semicontinuous on \((0,\infty)\).
\end{lemma}

\begin{proof}
Let \(t_n\to t>0\). If
\[
\liminf_{n\to\infty}
\frac{m(t_n)}
{t_n^{p_\alpha}}
=
+\infty,
\]
there is nothing to prove. Otherwise, passing to a subsequence if necessary,
we may assume that
\begin{equation}\label{eq:mt-liminf-sequence}
	\frac{m(t_n)}
	{t_n^{p_\alpha}}
	\longrightarrow
	\liminf_{k\to\infty}
	\frac{m(t_k)}
	{t_k^{p_\alpha}}
	<\infty.
\end{equation}
For each \(n\), let \(u_n\in\Sigma\) be a minimizer of \(m(t_n)\). Then
\[
	\norm{u_n}_{\ell^\alpha(\mathbb Z^N)}=t_n,
	\qquad
	J(u_n)=m(t_n).
\]
Since \(t_n\to t>0\), the sequence
\(\{t_n^{p_\alpha}\}\) is bounded.
Since the quotient in \eqref{eq:mt-liminf-sequence} converges to a
finite number and \(t_n^{p_\alpha}
\to
t^{p_\alpha}>0\), it follows that
\(\sup_n m(t_n)<\infty\). Since
\(J(u_n)=m(t_n)\), the coercivity of \(J\) implies that
\(\{u_n\}\) is bounded in
\(D^{1,2}(\mathbb Z^N)\). Passing to a subsequence,
\[
	u_n\rightharpoonup u
	\qquad \text{as }n\to\infty
	\qquad\text{in }D^{1,2}(\mathbb Z^N)
\]
for some \(u\in X\). By Lemma~\ref{lem:compact-symmetric-embedding}, we have
\[
	u_n\longrightarrow u
	\qquad \text{as }n\to\infty
	\qquad\text{in }\ell^\alpha(\mathbb Z^N)
\]
and therefore
\(\norm{u}_{\ell^\alpha(\mathbb Z^N)}=t\).
The weak lower semicontinuity of \(J\) yields
\[
\frac{m(t)}{t^{p_\alpha}}
\leq
\frac{J(u)}{t^{p_\alpha}}
\leq
\liminf_{n\to\infty}
\frac{J(u_n)}{t_n^{p_\alpha}}
=
\liminf_{n\to\infty}
\frac{m(t_n)}{t_n^{p_\alpha}}.
\]
\end{proof}

\subsection{The \texorpdfstring{\(Q_1\)}{Q1}-interpolation estimates}

We next introduce the \(Q_1\)-interpolation, which provides the
discrete-to-continuum comparison needed to estimate
\(\frac{m(t)}{t^{p_\alpha}}\) as \(t\to\infty\).

\begin{definition}
Let \(u:\mathbb Z^N\longrightarrow\mathbb R\). For \(x\in\mathbb Z^N\), set
\[
	Q_x:=x+[0,1)^N.
\]

\begin{enumerate}
	\item[\textup{(i)}]
	The function \(Pu:\mathbb R^N\to\mathbb R\) is defined by
	\[
		Pu(x+\theta):=u(x),
		\qquad
		x\in\mathbb Z^N,\quad \theta\in[0,1)^N.
	\]

	\item[\textup{(ii)}]
	The \(Q_1\) interpolant \(Iu:\mathbb R^N\to\mathbb R\) is defined by
	\[
		Iu(x+\theta)
		:=
		\sum_{\nu\in\{0,1\}^N}
		\lambda_\nu(\theta)\,u(x+\nu),
		\qquad
		x\in\mathbb Z^N,\quad \theta\in[0,1]^N,
	\]
	where
	\[
		\lambda_\nu(\theta)
		:=
		\prod_{j=1}^N
		\theta_j^{\nu_j}
		(1-\theta_j)^{1-\nu_j},
		\qquad
		\nu\in\{0,1\}^N.
	\]
\end{enumerate}
These local definitions agree on common faces of neighboring cubes, so \(Iu\) is well defined on \(\mathbb R^N\).
\end{definition}

\begin{example}
Consider the function
\[
	u(0)=1,\qquad
	u(1)=3,\qquad
	u(2)=2,\qquad
	u(3)=4.
\]
The functions \(Iu\) and \(Pu\) are illustrated below.

\begin{center}
	\begin{minipage}{0.47\textwidth}
		\centering
		\begin{tikzpicture}[x=1.15cm,y=0.75cm]
			\draw[->] (-0.25,0) -- (3.55,0)
				node[right] {\(x\)};
			\draw[->] (0,-0.25) -- (0,4.65)
				node[above] {\(u\)};

			\foreach \x in {0,1,2,3}
			{
				\draw (\x,0.08) -- (\x,-0.08)
					node[below=3pt] {\(\x\)};
			}

			\foreach \y in {1,2,3,4}
			{
				\draw (0.08,\y) -- (-0.08,\y)
					node[left=3pt] {\(\y\)};
			}

			\foreach \x in {1,2,3}
			{
				\draw[densely dashed] (\x,0) -- (\x,4.3);
			}

			\draw[very thick]
				(0,1) --
				(1,3) --
				(2,2) --
				(3,4);

			\fill (0,1) circle (2.2pt);
			\fill (1,3) circle (2.2pt);
			\fill (2,2) circle (2.2pt);
			\fill (3,4) circle (2.2pt);
		\end{tikzpicture}

		\vspace{0.25cm}

		\textup{(a) \(Iu\)}
	\end{minipage}
	\hfill
	\begin{minipage}{0.47\textwidth}
		\centering
		\begin{tikzpicture}[x=1.05cm,y=0.75cm]
			\draw[->] (-0.25,0) -- (4.35,0)
				node[right] {\(x\)};
			\draw[->] (0,-0.25) -- (0,4.65)
				node[above] {\(u\)};

			\foreach \x in {0,1,2,3,4}
			{
				\draw (\x,0.08) -- (\x,-0.08)
					node[below=3pt] {\(\x\)};
			}

			\foreach \y in {1,2,3,4}
			{
				\draw (0.08,\y) -- (-0.08,\y)
					node[left=3pt] {\(\y\)};
			}

			\foreach \x in {1,2,3,4}
			{
				\draw[densely dashed] (\x,0) -- (\x,4.3);
			}

			\draw[very thick] (0,1) -- (1,1);
			\draw[very thick] (1,3) -- (2,3);
			\draw[very thick] (2,2) -- (3,2);
			\draw[very thick] (3,4) -- (4,4);

			\fill (0,1) circle (2.2pt);
			\fill (1,3) circle (2.2pt);
			\fill (2,2) circle (2.2pt);
			\fill (3,4) circle (2.2pt);

			\draw[fill=white] (1,1) circle (2.2pt);
			\draw[fill=white] (2,3) circle (2.2pt);
			\draw[fill=white] (3,2) circle (2.2pt);
			\draw[fill=white] (4,4) circle (2.2pt);
		\end{tikzpicture}

		\vspace{0.25cm}

		\textup{(b) \(Pu\)}
	\end{minipage}
\end{center}
\end{example}

To apply the Sobolev inequality, we first show that
\(Iu\in D^{1,p}(\mathbb R^N)\) for any \(u\in D^{1,p}(\mathbb Z^N)\) and establish the following gradient
estimate.

\begin{lemma}\label{lem:Q1-gradient-estimate}
Let \(2\leq p<N\). Then, for any \(u\in D^{1,p}(\mathbb Z^N)\),
\(Iu\in D^{1,p}(\mathbb R^N)\) and
\[
	\int_{\mathbb R^N}
	\abs{\nabla Iu(\xi)}^p\,d\xi
	\leq
	\frac{N^{\frac p2-1}}{2}
	\sum_{x\in\mathbb Z^N}\sum_{y\sim x}
	\abs{\nabla_{xy}u}^p.
\]
\end{lemma}

\begin{proof}
Fix \(i\in\{1,\ldots,N\}\). For
\(\theta\in[0,1]^N\) and
\(\mu\in\{0,1\}^N\) with \(\mu_i=0\), set
\[
	\lambda_\mu^{(i)}(\theta)
	:=
	\prod_{\substack{1\leq j\leq N\\ j\neq i}}
	\theta_j^{\mu_j}(1-\theta_j)^{1-\mu_j}.
\]
Differentiating with respect to \(\theta_i\), we obtain
\[
	\partial_i Iu(x+\theta)
	=
	\sum_{\substack{\mu\in\{0,1\}^N\\ \mu_i=0}}
	\lambda_\mu^{(i)}(\theta)
	\nabla_{x+\mu,x+\mu+e_i}u,
\]
where
\[
	\lambda_\mu^{(i)}(\theta)\geq0,
	\qquad
	\sum_{\substack{\mu\in\{0,1\}^N\\ \mu_i=0}}
	\lambda_\mu^{(i)}(\theta)=1.
\]
By Jensen's inequality,
\[
	\abs{\partial_i Iu(x+\theta)}^p
	\leq
	\sum_{\substack{\mu\in\{0,1\}^N\\ \mu_i=0}}
	\lambda_\mu^{(i)}(\theta)
	\abs{\nabla_{x+\mu,x+\mu+e_i}u}^p.
\]
Since
\[
	\int_{[0,1]^N}\lambda_\mu^{(i)}(\theta)\,d\theta
	=
	\frac1{2^{N-1}},
\]
we obtain
\[
\begin{aligned}
	\int_{\mathbb R^N}
	\abs{\partial_i Iu(\xi)}^p\,d\xi
	&=
	\sum_{x\in\mathbb Z^N}
	\int_{Q_x}
	\abs{\partial_i Iu(\xi)}^p\,d\xi
	\\
	&\leq
	\frac1{2^{N-1}}
	\sum_{x\in\mathbb Z^N}
	\sum_{\substack{\mu\in\{0,1\}^N\\ \mu_i=0}}
	\abs{\nabla_{x+\mu,x+\mu+e_i}u}^p
	\\
	&=
	\sum_{x\in\mathbb Z^N}
	\abs{\nabla_{x,x+e_i}u}^p.
\end{aligned}
\]
Moreover, since \(p\geq2\), by
\[
	\left(\sum_{i=1}^N a_i\right)^{p/2}
	\leq
	N^{\frac p2-1}
	\sum_{i=1}^N a_i^{p/2},
	\qquad a_i\geq0,
\]
we have
\[
	\abs{\nabla Iu}^p
	=
	\left(
		\sum_{i=1}^N\abs{\partial_i Iu}^2
	\right)^{\frac p2}
	\leq
	N^{\frac p2-1}
	\sum_{i=1}^N\abs{\partial_i Iu}^p.
\]
Therefore,
\[
\begin{aligned}
	\int_{\mathbb R^N}\abs{\nabla Iu(\xi)}^p\,d\xi
	&\leq
	N^{\frac p2-1}
	\sum_{i=1}^N
	\int_{\mathbb R^N}
	\abs{\partial_i Iu(\xi)}^p\,d\xi
	\\
	&\leq
	N^{\frac p2-1}
	\sum_{i=1}^N\sum_{x\in\mathbb Z^N}
	\abs{\nabla_{x,x+e_i}u}^p
	\\
	&=
	\frac{N^{\frac p2-1}}{2}
	\sum_{x\in\mathbb Z^N}\sum_{y\sim x}
	\abs{\nabla_{xy}u}^p.
\end{aligned}
\]
Set
\[
	p^*:=\frac{Np}{N-p}.
\]
By the discrete Sobolev inequality,
\(u\in\ell^{p^*}(\mathbb Z^N)\). For
\(x\in\mathbb Z^N\) and \(\theta\in[0,1]^N\), Jensen's inequality
gives
\[
	\abs{Iu(x+\theta)}^{p^*}
	\leq
	\sum_{\nu\in\{0,1\}^N}
	\lambda_\nu(\theta)
	\abs{u(x+\nu)}^{p^*}.
\]
Integrating over \([0,1]^N\) and summing over
\(x\in\mathbb Z^N\), we obtain
\[
	\norm{Iu}_{L^{p^*}(\mathbb R^N)}^{p^*}
	\leq
	\norm{u}_{\ell^{p^*}(\mathbb Z^N)}^{p^*}
	<\infty.
\]
Thus \(Iu\in L^{p^*}(\mathbb R^N)\), and hence
\(Iu\in D^{1,p}(\mathbb R^N)\).
\end{proof}

If
\[
\liminf_{t\to\infty}\frac{m(t)}
{t^{p_\alpha}}=+\infty,
\]
then the case \(t\to+\infty\) is already excluded in the final minimization
argument. Thus only the case of a finite lower limit requires further
analysis.

The preceding gradient estimate, combined with the continuous Sobolev
inequality, gives a lower bound involving
\(\|Iu_n\|_{L^\alpha(\mathbb R^N)}\). To compare this quantity with
\(t_n=\|u_n\|_{\ell^\alpha(\mathbb Z^N)}\), we prove the following
asymptotic relation.

\begin{lemma}\label{lem:interpolant-norm-asymptotics}
Let \(\alpha>2^*\).
Assume that
\[
\liminf_{t\to\infty}\frac{m(t)}{t^{p_\alpha}}<\infty.
\]
Let \(t_n\to\infty\) satisfy
\[
\frac{m(t_n)}{t_n^{p_\alpha}}
\longrightarrow
\liminf_{t\to\infty}\frac{m(t)}{t^{p_\alpha}},
\]
and let \(u_n\in X\) be such that
\[
	\norm{u_n}_{\ell^\alpha(\mathbb Z^N)}=t_n,
	\qquad
	J(u_n)=m(t_n).
\]
Then
\[
	\frac{\norm{Iu_n}_{L^\alpha(\mathbb R^N)}}{t_n}
	\longrightarrow1
	\qquad\text{as }n\to\infty.
\]
\end{lemma}

\begin{proof}
We compare \(Iu_n\) with the piecewise constant extension \(Pu_n\).
First,
\[
\begin{aligned}
	\norm{Pu_n}_{L^\alpha(\mathbb R^N)}^\alpha
	&=
	\sum_{x\in\mathbb Z^N}
	\int_{Q_x}\abs{Pu_n(\xi)}^\alpha\,d\xi\\
	&=
	\sum_{x\in\mathbb Z^N}\abs{u_n(x)}^\alpha
	=t_n^\alpha,
\end{aligned}
\]
and hence \(\norm{Pu_n}_{L^\alpha(\mathbb R^N)}=t_n\).

For \(x\in\mathbb Z^N\) and \(\theta\in[0,1]^N\),
\[
	Iu_n(x+\theta)-Pu_n(x+\theta)
	=
	\sum_{\nu\in\{0,1\}^N}
	\lambda_\nu(\theta)
	\bigl(u_n(x+\nu)-u_n(x)\bigr).
\]
Jensen's inequality gives
\[
\abs{Iu_n(x+\theta)-Pu_n(x+\theta)}^{p_\alpha}
\leq
\sum_{\nu\in\{0,1\}^N}
\lambda_\nu(\theta)
\abs{u_n(x+\nu)-u_n(x)}^{p_\alpha}.
\]
Notice that \(x+\nu\) need not be adjacent to \(x\). We therefore
connect \(x\) to \(x+\nu\) through the vertices
\(z_0,z_1,\ldots,z_N\) defined as follows.
For \(\nu\in\{0,1\}^N\), set
\[
	z_0:=x,
	\qquad
	z_k:=x+\sum_{j=1}^k\nu_j e_j,
	\quad k=1,\ldots,N.
\]
Then \(z_N=x+\nu\) and
\[
	u_n(x+\nu)-u_n(x)
	=
	\sum_{k=1}^N
	\nu_k
	\nabla_{z_{k-1},\,z_{k-1}+e_k}u_n.
\]
Using
\[
\left|\sum_{k=1}^N a_k\right|^{p_\alpha}
\leq
N^{p_\alpha-1}
\sum_{k=1}^N|a_k|^{p_\alpha},
\]
we obtain
\[
\abs{u_n(x+\nu)-u_n(x)}^{p_\alpha}
\leq
N^{p_\alpha-1}
\sum_{k=1}^N
\nu_k
\abs{\nabla_{z_{k-1},\,z_{k-1}+e_k}u_n}^{p_\alpha}.
\]
Notice that, for fixed \(k\) and \(\nu\) with \(\nu_k=1\),
\[
\begin{aligned}
	\sum_{x\in\mathbb Z^N}
	\abs{\nabla_{z_{k-1},\,z_{k-1}+e_k}u_n}^{p_\alpha}
	&=
	\sum_{x\in\mathbb Z^N}
	\abs{
		\nabla_{
			x+\sum_{j=1}^{k-1}\nu_j e_j,\,
			x+\sum_{j=1}^{k-1}\nu_j e_j+e_k
		}
		u_n
	}^{p_\alpha}
	\\
	&=
	\sum_{x\in\mathbb Z^N}
	\abs{\nabla_{x,x+e_k}u_n}^{p_\alpha}.
\end{aligned}
\]
Since
\[
	\int_{[0,1]^N}\lambda_\nu(\theta)\,d\theta=2^{-N},
\]
and exactly \(2^{N-1}\) vectors
\(\nu\in\{0,1\}^N\) satisfy \(\nu_k=1\), summing over
\(\nu\in\{0,1\}^N\) gives the factor \(2^{N-1}\).
Moreover, notice that
\[
\sum_{x\in\mathbb Z^N}\sum_{y\sim x}
|\nabla_{xy}u_n|^{p_\alpha}
=
2\sum_{i=1}^N\sum_{x\in\mathbb Z^N}
|\nabla_{x,x+e_i}u_n|^{p_\alpha}.
\]
Hence, we obtain
\[
\begin{aligned} 
	\norm{Iu_n-Pu_n}_{L^{p_\alpha}(\mathbb R^N)}^{p_\alpha}
	&\leq 
	\frac{N^{p_\alpha-1}}{2}
	\sum_{i=1}^N\sum_{x\in\mathbb Z^N} 
	\abs{\nabla_{x,x+e_i}u_n}^{p_\alpha}\\
	&= 
	\frac{N^{p_\alpha-1}}{4}
	\sum_{x\in\mathbb Z^N}\sum_{y\sim x} 
	\abs{\nabla_{xy}u_n}^{p_\alpha}.
\end{aligned} 
\]

Since \(u_n\in X\), for every \(\nu\in\{0,1\}^N\),
\[
\begin{aligned}
	\abs{u_n(x+\nu)-u_n(x)}
	&\leq
	\sum_{k=1}^N
	\nu_k
	\abs{\nabla_{z_{k-1},\,z_{k-1}+e_k}u_n}
	\\
	&\leq N.
\end{aligned}
\]
Therefore,
\[
\begin{aligned}
	\abs{Iu_n(x+\theta)-Pu_n(x+\theta)}
	&\leq
	\sum_{\nu\in\{0,1\}^N}
	\lambda_\nu(\theta)
	\abs{u_n(x+\nu)-u_n(x)}
	\\
	&\leq N,
\end{aligned}
\]
and hence
\[
	\norm{Iu_n-Pu_n}_{L^\infty(\mathbb R^N)}\leq N.
\]
By H\"older's inequality,
\[
\begin{aligned} 
	\norm{Iu_n-Pu_n}_{L^\alpha(\mathbb R^N)}^\alpha 
	&\leq 
	\norm{Iu_n-Pu_n}_{L^\infty(\mathbb R^N)}^{\alpha-p_\alpha}
	\norm{Iu_n-Pu_n}_{L^{p_\alpha}(\mathbb R^N)}^{p_\alpha}\\
	&\leq 
	\frac{N^{\alpha-1}}{4} 
	\sum_{x\in\mathbb Z^N}\sum_{y\sim x} 
	\abs{\nabla_{xy}u_n}^{p_\alpha}
	\leq 
	N^{\alpha-1}m(t_n), 
\end{aligned} 
\]
where in the last inequality we used
\(\kappa_{p_\alpha}\geq\frac12\), which follows from
\(1-\sqrt{1-s}\geq s/2\) and \(s^{p_\alpha/2}\leq s\) for
\(s\in[0,1]\) and \(p_\alpha\geq2\). In particular,
\(Iu_n-Pu_n\in L^\alpha(\mathbb R^N)\), and hence
\(Iu_n\in L^\alpha(\mathbb R^N)\).
By the choice of \(t_n\), there exists \(C>0\) such that
\[
m(t_n)\leq Ct_n^{p_\alpha}
\]
for all sufficiently large \(n\). Therefore, since \(p_\alpha<\alpha\),
\[
\left(
\frac{\norm{Iu_n-Pu_n}_{L^\alpha(\mathbb R^N)}}{t_n}
\right)^\alpha
\leq
CN^{\alpha-1}t_n^{p_\alpha-\alpha}
\longrightarrow0
\qquad\text{as }n\to\infty.
\]
Hence,
\[
	\left|
		\frac{\norm{Iu_n}_{L^\alpha(\mathbb R^N)}}{t_n}-1
	\right|
	\leq
	\frac{\norm{Iu_n-Pu_n}_{L^\alpha(\mathbb R^N)}}{t_n}
	\longrightarrow0
	\qquad\text{as }n\to\infty.
\]
This proves the desired asymptotic relation.
\end{proof}

\subsection{Proof of Theorem~\ref{achieved}}
To exclude minimizing sequences with \(t\to\infty\), we need to prove
that
\[
C_{N,\alpha}
<
\liminf_{t\to\infty}
\frac{m(t)}{t^{p_\alpha}}.
\]
The underlying mechanism is a separation between the discrete critical
Sobolev level and the continuous Sobolev level that appears in the limit
\(t\to\infty\). The strict comparison \(\mathcal S_2<S_2^{\mathbb R^N}\) was recently
established in \cite{HeJi2026}, but this result is not used in the proof
below. Instead, we obtain the required separation directly and
quantitatively: a discrete test function yields, for \(\alpha>2^*\)
sufficiently close to \(2^*\),
\[
C_{N,\alpha}<\frac{2N-1}{2},
\]
whereas the \(Q_1\) comparison with the Euclidean Sobolev inequality
gives
\[
\liminf_{t\to\infty}\frac{m(t)}{t^{p_\alpha}}
\geq \frac12S_2^{\mathbb R^N}-\eta.
\]
Since \(S_2^{\mathbb R^N}>2N-1\), choosing \(\eta>0\) sufficiently small
produces the desired strict separation. The next three lemmas make
this argument precise.

\begin{lemma}\label{lem:strict-upper-bound-near-critical}
Let \(N\geq3\). Then there exists \(\delta>0\) such that
\[
	C_{N,\alpha}<\frac{2N-1}{2}
\]
for every
\[
	2^*<\alpha<2^*+\delta.
\]
\end{lemma}

\begin{proof}
For \(2<p<N\), set
\(\alpha(p):=\frac{Np}{N-p}\) and
\(u_p(x):=(p-2)v(x)\), where
\[
	v(x)
	:=
	\begin{cases}
		1, & x=0,\\[1mm]
		\dfrac{1}{2N}, & x=\pm e_i,\quad i=1,\ldots,N,\\[2mm]
		0, & \text{otherwise}.
	\end{cases}
\]
For \(p>2\) sufficiently close to \(2\), we have \(u_p\in X\), since
\[
	\max_{x\sim y}\abs{\nabla_{xy}u_p}
	=
	\frac{(p-2)(2N-1)}{2N}
	\leq1.
\]

By
\(\frac{N\alpha(p)}{N+\alpha(p)}=p\)
and the definition of \(C_{N,\alpha(p)}\), we have
\[
	C_{N,\alpha(p)}
	\leq
	\frac{J(u_p)}
	{\norm{u_p}_{\ell^{\alpha(p)}(\mathbb Z^N)}^p}.
\]
There are \(2N\) unoriented center--neighbor edges, each with
edge difference \((p-2)(2N-1)/(2N)\), and
\(2N(2N-1)\) unoriented neighbor--exterior edges, each with edge difference
\((p-2)/(2N)\). Since the factor \(1/2\) in \(J\) removes the double
counting of ordered edges, we obtain
\[
\begin{aligned}
	J(u_p)
	={}&
	2N\left(
		1-\sqrt{
			1-\frac{(p-2)^2(2N-1)^2}{4N^2}
		}
	\right)
	\\
	&+
	2N(2N-1)\left(
		1-\sqrt{
			1-\frac{(p-2)^2}{4N^2}
		}
	\right).
\end{aligned}
\]
Set
\[
	A_p
	:=
	2N
	\frac{
	1-\sqrt{
	1-\frac{(p-2)^2(2N-1)^2}{4N^2}
	}
	}{
	(p-2)^2
	}
	+
	2N(2N-1)
	\frac{
	1-\sqrt{
	1-\frac{(p-2)^2}{4N^2}
	}
	}{
	(p-2)^2
	}.
\]
Then
\[
	J(u_p)=(p-2)^2A_p,
\]
while
\[
	\norm{u_p}_{\ell^{\alpha(p)}(\mathbb Z^N)}^p
	=
	(p-2)^p
	\left[
		1+(2N)^{1-\frac{Np}{N-p}}
	\right]^{\frac{N-p}{N}}.
\]
Consequently,
\[
	C_{N,\alpha(p)}
	\leq
	\frac{
	(p-2)^{2-p}A_p
	}{
	\left[
	1+(2N)^{1-\frac{Np}{N-p}}
	\right]^{\frac{N-p}{N}}
	}.
\]

Using
\[
	1-\sqrt{1-s}\sim\frac{s}{2}
	\qquad\text{as }s\to0^+,
\]
and
\[
	(p-2)^{2-p}\longrightarrow1
	\qquad\text{as }p\to2^+,
\]
we obtain
\[
	\limsup_{p\to2^+}C_{N,\alpha(p)}
	\leq
	\frac{2N-1}
	{2\left[
			1+(2N)^{1-2^*}
		\right]^{\frac{N-2}{N}}
	}
	<\frac{2N-1}{2}.
\]
By the preceding limit, there exists \(\varepsilon\in(0,N-2)\) such that
\[
	C_{N,\alpha(p)}<\frac{2N-1}{2},
	\qquad 2<p<2+\varepsilon.
\]
Since \(\alpha(p)=\frac{Np}{N-p}\) is continuous and strictly
increasing, setting
\[
	\delta:=\alpha(2+\varepsilon)-2^*>0,
\]
we obtain
\[
	C_{N,\alpha}<\frac{2N-1}{2},
	\qquad 2^*<\alpha<2^*+\delta.
\]
\end{proof}
For later use, let \(S_p^{\mathbb R^N}\) denote the optimal constant in the Sobolev inequality on
	\(\mathbb R^N\).
\begin{lemma}\label{lem:lower-bound-at-infinity}
Let \(N\geq3\). For every \(\eta>0\), there exists
\(\varepsilon>0\) such that, for every \(\alpha\in(2^*,2^*+\varepsilon)\),
\[
\liminf_{t\to\infty}
\frac{m(t)}
{t^{p_\alpha}}
\geq \frac12 S_2^{\mathbb R^N}-\eta.
\]
\end{lemma}

\begin{proof}
Fix \(\eta>0\) and \(\alpha>2^*\).

If \(\liminf_{t\to\infty}m(t)/t^{p_\alpha}=+\infty\), the conclusion is immediate. Otherwise, choose \(t_n\to\infty\) such
that
\[
\frac{m(t_n)}{t_n^{p_\alpha}}
\longrightarrow
\liminf_{t\to\infty}\frac{m(t)}{t^{p_\alpha}},
\]
and let \(u_n\in X\) satisfy
\[
	\norm{u_n}_{\ell^\alpha(\mathbb Z^N)}=t_n,
	\qquad
	J(u_n)=m(t_n).
\]

Because \(u_n\in X\), we have \(|\nabla_{xy}u_n|\leq1\). Since \(p_\alpha>2\),
\[
\sum_{x\in\mathbb Z^N}\sum_{y\sim x}|\nabla_{xy}u_n|^{p_\alpha}
\leq
\sum_{x\in\mathbb Z^N}\sum_{y\sim x}|\nabla_{xy}u_n|^2<\infty.
\]
Moreover \(u_n\in\ell^\alpha(\mathbb Z^N)\) and \(\alpha=p_\alpha^*\). Hence
Lemma~\ref{lem:finite-energy-characterization} gives
\(u_n\in D^{1,p_\alpha}(\mathbb Z^N)\), so
Lemma~\ref{lem:Q1-gradient-estimate} is applicable.

By \eqref{eq:kappa-lower-bound} and
Lemma~\ref{lem:Q1-gradient-estimate},
\[
\begin{aligned}
	\frac{m(t_n)}{t_n^{p_\alpha}}
	&\geq
	\frac{\kappa_{p_\alpha}}{2t_n^{p_\alpha}}
	\sum_{x\in\mathbb Z^N}\sum_{y\sim x}
	\abs{\nabla_{xy}u_n}^{p_\alpha}
	\\
	&\geq
	\frac{\kappa_{p_\alpha}}{N^{\frac{p_\alpha}2-1}}
	\frac{
		\displaystyle
		\int_{\mathbb R^N}\abs{\nabla Iu_n}^{p_\alpha}\,dx
	}{
		t_n^{p_\alpha}
	}
	\\
	&\geq
	\frac{\kappa_{p_\alpha}
		S_{p_\alpha}^{\mathbb R^N}}
	{N^{\frac{p_\alpha}2-1}}
	\left(
	\frac{
		\norm{Iu_n}_{L^\alpha(\mathbb R^N)}
	}{
		t_n
	}
	\right)^{p_\alpha}.
\end{aligned}
\]
Lemma~\ref{lem:interpolant-norm-asymptotics} gives
\[
	\frac{\norm{Iu_n}_{L^\alpha(\mathbb R^N)}}{t_n}
	\longrightarrow1
	\qquad\text{as }n\to\infty,
\]
and therefore
\[
\liminf_{t\to\infty}\frac{m(t)}{t^{p_\alpha}}
\geq
\frac{\kappa_{p_\alpha}}
{N^{\frac{p_\alpha}2-1}}
S_{p_\alpha}^{\mathbb R^N}.
\]

As \(\alpha\to(2^*)^+\), we have \(p_\alpha\to2^+\).
By the formula for \(\kappa_p\) recorded in Section~\ref{2}
and the explicit formula for \(S_{p_\alpha}^{\mathbb R^N}\) in
\cite{Talenti1976}, we have
\(\kappa_{p_\alpha}\to\frac12\) and
\(S_{p_\alpha}^{\mathbb R^N}\to S_2^{\mathbb R^N}\)
as \(\alpha\to(2^*)^+\), and hence
\[
\frac{\kappa_{p_\alpha}}
{N^{\frac{p_\alpha}2-1}}
S_{p_\alpha}^{\mathbb R^N}
\longrightarrow \frac12 S_2^{\mathbb R^N}.
\]
Thus there exists \(\varepsilon>0\) such that, for every
\(\alpha\in(2^*,2^*+\varepsilon)\),
\[
\liminf_{t\to\infty}
\frac{m(t)}
{t^{p_\alpha}}
\geq \frac12 S_2^{\mathbb R^N}-\eta.
\]
\end{proof}

\begin{lemma}\label{lem:S2-lower-bound}
For every \(N\geq3\), \(S_2^{\mathbb R^N}>2N-1\).
\end{lemma}

\begin{proof}
By the explicit formula for \(S_p^{\mathbb R^N}\) in \cite{Talenti1976} with \(p=2\),
we have
\[	S_2^{\mathbb R^N}
=
N(N-2)\pi
\left(
\frac{\Gamma(N/2)}{\Gamma(N)}
\right)^{2/N}.\]
For \(N=2m\),
\[
	\frac{\Gamma(2m)}{\Gamma(m)}
	=
	m(m+1)\cdots(2m-1)
	<
	(2m)^m.
\]
For \(N=2m+1\), we use
\(m!\leq((m+1)/2)^m\), which follows from the arithmetic--geometric
mean inequality applied to \(1,\ldots,m\). Thus
\[
\begin{aligned}
	\frac{\Gamma(2m+1)}{\Gamma(m+1/2)}
	&=
	\frac{4^m m!}{\sqrt{\pi}}
	\leq
	\frac{[2(m+1)]^m}{\sqrt{\pi}}\\
	&=
	\frac{(2m+1)^m}{\sqrt{\pi}}
	\left(1+\frac{1}{2m+1}\right)^m
	<
	(2m+1)^{m+1/2}.
\end{aligned}
\]
Here we used
\[
	\left(1+\frac{1}{2m+1}\right)^m
	<
	e^{\frac{m}{2m+1}}
	<
	\sqrt e
	<
	\sqrt{\pi(2m+1)}.
\]
Thus
\[
	\frac{\Gamma(N)}{\Gamma(N/2)}<N^{N/2},
\]
and hence
\[
	S_2^{\mathbb R^N}>\pi(N-2).
\]
For \(N\geq5\),
\[
	S_2^{\mathbb R^N}>\pi(N-2)>3(N-2)\geq2N-1.
\]
For \(N=3,4\),
\[
	S_2^{\mathbb R^N}\big|_{N=3}
	=
	\frac{3\pi^{4/3}}{2^{4/3}}>5,
	\qquad
	S_2^{\mathbb R^N}\big|_{N=4}
	=
	\frac{8\pi}{\sqrt6}>7.
\]
\end{proof}

We are now in a position to complete the proof.

\begin{proof}
By Lemma~\ref{lem:S2-lower-bound}, choose \(\eta>0\) such that
\[
	\frac{2N-1}{2}<\frac12S_2^{\mathbb R^N}-\eta.
\]
By Lemma~\ref{lem:strict-upper-bound-near-critical}, there exists
\(\varepsilon_1>0\) such that
\[
	C_{N,\alpha}<\frac{2N-1}{2}
\]
for every \(2^*<\alpha<2^*+\varepsilon_1\). By
Lemma~\ref{lem:lower-bound-at-infinity}, there exists
\(\varepsilon_2>0\) such that
\[
\frac12S_2^{\mathbb R^N}-\eta
\leq
\liminf_{t\to\infty}
\frac{m(t)}
{t^{p_\alpha}}
\]
for every \(2^*<\alpha<2^*+\varepsilon_2\). Set
\[
	\varepsilon_0
	:=
	\min\{\varepsilon_1,\varepsilon_2\}.
\]
Then, for every \(2^*<\alpha<2^*+\varepsilon_0\),
\[
C_{N,\alpha}
<
\frac{2N-1}{2}
<
\frac12S_2^{\mathbb R^N}-\eta
\leq
\liminf_{t\to\infty}
\frac{m(t)}{t^{p_\alpha}}.
\]
Fix such an \(\alpha\). The preceding strict separation excludes minimizing sequences with
	\(t\to\infty\). To exclude \(t\to0^+\), we use the discrete
Sobolev inequality. For every \(u\in X\),
\eqref{eq:discrete-sobolev-inequality} gives
\[
	J(u)
	\geq
	\frac12\norm{u}_{D^{1,2}(\mathbb Z^N)}^2
	\geq
	c_N\norm{u}_{\ell^{2^*}(\mathbb Z^N)}^2
	\geq
	c_N\norm{u}_{\ell^\alpha(\mathbb Z^N)}^2,
\]
where the last inequality follows from
\(\ell^{2^*}(\mathbb Z^N)\subset\ell^\alpha(\mathbb Z^N)\) for
\(\alpha>2^*\).
Thus \(m(t)\geq c_Nt^2\), and
\begin{equation}\label{eq:mt-small-t}
	\frac{m(t)}{t^{p_\alpha}}
	\geq
	c_Nt^{2-p_\alpha}
	\longrightarrow+\infty
	\qquad\text{as }t\to0^+.
\end{equation}

By Proposition~\ref{prop:quotient-constrained-relation},
\[
C_{N,\alpha}
=
\inf_{t>0}\frac{m(t)}{t^{p_\alpha}}.
\]
We now show that this infimum is attained. Let
\(\{t_n\}\subset(0,\infty)\) satisfy
\[
\frac{m(t_n)}{t_n^{p_\alpha}}
\longrightarrow C_{N,\alpha}
\qquad\text{as }n\to\infty.
\]
By \eqref{eq:mt-small-t} and
\[
C_{N,\alpha}
<
\liminf_{t\to\infty}\frac{m(t)}
{t^{p_\alpha}},
\]
there exist constants \(c,C>0\) such that
\[
	c\leq t_n\leq C
\]
for all sufficiently large \(n\). Hence, passing to a subsequence if
necessary,
\[
	t_n\longrightarrow t_0
	\qquad\text{for some }t_0\in(0,\infty).
\]
By Lemma~\ref{lem:mt-lower-semicontinuity},
\[
\frac{m(t_0)}{t_0^{p_\alpha}}
\leq
\liminf_{n\to\infty}
\frac{m(t_n)}{t_n^{p_\alpha}}
=
C_{N,\alpha}.
\]
By Proposition~\ref{prop:quotient-constrained-relation},
\[
C_{N,\alpha}
\leq
\frac{m(t_0)}{t_0^{p_\alpha}}.
\]
Hence
\[
C_{N,\alpha}
=
\frac{m(t_0)}{t_0^{p_\alpha}}.
\]
By Proposition~\ref{prop:constrained-minimizer}, there exists a Schwarz
symmetric function \(u_0\in X\) such that
\[
	\norm{u_0}_{\ell^\alpha(\mathbb Z^N)}=t_0
	\qquad\text{and}\qquad
	J(u_0)=m(t_0).
\]
Therefore,
\[
C_{N,\alpha}
=
\frac{
	J(u_0)
}{
	\norm{u_0}_{\ell^\alpha(\mathbb Z^N)}^{p_\alpha}
},
\]
and thus \(C_{N,\alpha}\) is attained.
\end{proof}

\subsection*{Conflict of interest}

The authors declare no conflict of interest.

\subsection*{Ethics approval}
 Not applicable.

\subsection*{Data Availability Statements}
Data sharing not applicable to this article as no datasets were generated or analysed during the current study.

\subsection*{Acknowledgements}
 C. Ji was supported by National Natural Science Foundation of China (No. 12571117).

\end{document}